\documentclass[final]{amsart}

\usepackage{graphicx,pst-eps,epstopdf}
\usepackage{lscape}
\usepackage{indentfirst}
\usepackage{latexsym}
\usepackage{amsmath, amsfonts, amssymb,mathrsfs}
\usepackage{multirow}
\usepackage{subfigure,pstricks,pst-node}
\usepackage{pst-eps,epstopdf}
\usepackage{enumerate}
\usepackage{verbatim}
\usepackage{setspace}
\usepackage{colortbl}
\usepackage{morefloats}
\usepackage{showkeys}

\newtheorem{Theorem}{Theorem}[section]
\newtheorem{Lemma}{Lemma}[section]

\newtheorem{algorithm}{Algorithm}[section]

\newtheorem{remark}{Remark}[section]

\newcommand{\A}{\mathcal{A}}

\newcommand{\B}{\mathcal{B}}

\newcommand{\E}{\mathcal{E}}
\newcommand{\K}{\mathcal{N}}
\newcommand{\R}{\mathcal{R}}

\newcommand{\I}{\mathcal{I}}

\newcommand{\Label}[1]{\label{#1}\mbox{\fbox{\bf #1}\quad}}
\renewcommand{\Label}[1]{\label{#1}}

\newcommand{\Null}[1]{\mathcal{N}(#1)}
\renewcommand{\ker}[1]{\Null}

\newcommand{\vu}{{\bf{u}}}

\newcounter{mnote}
\let\oldmarginpar\marginpar
\renewcommand\marginpar[1]{\-\oldmarginpar[\raggedleft\footnotesize #1]%
  {\raggedright\footnotesize #1}}
    
\begin{document}

\title[Scalable GAMG for Poisson]{A Scalable Auxiliary Space Preconditioner for High-Order Finite Element Methods}

\author{Young-Ju Lee}
\thanks{Department of Mathematics, Rutgers, The State University of New Jersey,  NJ, 08901, USA. Email: leeyoung@math.rutgers.edu. Supported in part by NSF DMS-0915028 and the Startup fund from Rutgers University.}        

\author{Wei Leng}
\thanks{Institute of Computational Mathematics, Chinese Academy of Sciences, Beijing 100190, China. Email: wleng@lsec.cc.ac.cn.}

\author{Chen-song Zhang}
\thanks{Institute of Computational Mathematics, Chinese Academy of Sciences, Beijing 100190, China. Email: zhangcs@lsec.cc.ac.cn. Supported in part by NSF DMS-{0915153}.}

\date{April/10/2012}

\maketitle

\begin{abstract}
In this paper, we revisit an auxiliary space preconditioning method proposed by Xu~[Computing 56, 1996], in which low-order finite element spaces are employed as auxiliary spaces for solving linear algebraic systems arising from high-order finite element discretizations. We provide a new convergence rate estimate and parallel implementation of the proposed algorithm. We show that this method is user-friendly and can play an important role in a variety of Poisson-based solvers for more challenging problems such as the Navier--Stokes equation. We investigate the performance of the proposed algorithm using the Poisson equation and the Stokes equation on 3D unstructured grids. Numerical results demonstrate the advantages of the proposed algorithm in terms of efficiency, robustness, and parallel scalability.
\end{abstract}




\section{Introduction}\label{sec:intro}

Iterative methods have been successfully applied to large-scale sparse linear systems arising from discretizations of partial differential equations (PDEs). Many linear systems of equations can be handled by preconditioned Krylov subspace methods \cite{Axelsson1994,Saad2003}. In fact, preconditioners play a crucial role in the convergence of iterative methods. The construction of an ``ideal" preconditioner depends on following basic, sometimes contradictory, guidelines: (1) {{Optimality and Robustness:}} The convergence rate of an appropriate iterative methods on the preconditioned system is uniform or nearly uniform, independent of mesh size or physical parameters; (2) {{Cost-effectiveness and Scalability:}} The computational costs and memory requirements of constructing and applying the preconditioning action are low and have good parallel scalability; and (3) {{User-friendliness:}} Users require little information and implementation is not difficult. 



The Poisson equation $-\Delta u = f$ and its variants arise in many applications. The geometric multigrid (GMG) method is one of the most efficient iterative methods for solving discrete Poisson or Poisson-like equations. A vast number of works have explored multigrid methods; references include the monographs and the survey papers~\cite{Brandt.A1984,Hackbusch.W1985,Bramble.J1993,Trottenberg2001}. Though the classical multigrid algorithm based on a geometric hierarchy can be an effective solver for a well-structured, it is usually very difficult to obtain such a hierarchy in practice. The algebraic multigrid (AMG) method~\cite{Brandt.A;McCormick.S;Ruge.J1984,Brandt.A1986,Stuben.K2000,Trottenberg2001,Brandt.A2002,Falgout2006}, on the other hand, requires minimal geometric information about the underlying problem and can sometimes be employed as a ``black-box'' iterative solver or preconditioner for other iterative methods. The version known as the Classical AMG~\cite{Brandt.A;McCormick.S;Ruge.J1984,Ruge1987} is used frequently and has been shown to be effective for a range of problems in practice. In an effort to render AMG methodologies more broadly applicable and to improve robustness, various versions have been developed; for example, see~\cite{Vanek.P;Mandel.J;Brezina.M1996,Xu.J;Zikatanov.L2003b,Muresan.Notay2008,Brannick.Falgout2010}.
%

AMG methods are readily applicable and potentially scalable for large 3D problems. Recently, parallel versions of multigrid methods have attracted a lot of attention (and will continue to do so) because of their fundamental role in modern computational mathematics and engineering; see \cite{Griebel.Metsch.Oeltz.Schweitzer:2005,Falgout.Jones.ea2005,Sterck.H;Falgout.R;Nolting.J;Yang.U2008,Bakera} and references therein for details. In this paper, we will not discuss parallelization and implementation of AMG. Throughout this paper, we employ the Parallel Modified Independent Set (PMIS) coarsening strategy~\cite{Sterck2006} and the Extend+i+cc interpolation~\cite{Sterck.H;Falgout.R;Nolting.J;Yang.U2008} in BoomerAMG of {\it hypre} package~\cite{hypre}, which has been numerically proven to be efficient and scalable~\cite{Bakera}.

Although AMG methods have been proven effective for many problems, it is important to note that generally the performance of the Classical AMG method deteriorates for high-order finite elements (see Shu, Sun and Xu~\cite{Shu.S;Sun.D;Xu.J2006} for the 2D test examples). In Table~\ref{tab:Poisson-hypre}, we show a simple numerical experiment. It is easy to see that for about the same degree of freedom ($10^7$) the convergence rate of the AMG method (PMIS and Extend+i+cc) deteriorates. Furthermore, the performance of AMG is very sensitive to the strength threshold $\theta$ in the coarsening procedure (see \S\ref{sec:numer} for details). 
On the other hand, it is clear that AMG can be also very effective for the discrete Poisson equations in relatively low order finite element spaces. 
\begin{table}[ht]
\renewcommand{\arraystretch}{1.2}
\centering\caption{\rm Number of iterations for the AMG preconditioned GMRES method. We solve the 3D Poisson equation with 64 processing cores (piecewise continuous Lagrangian finite element discretizations are applied, the stopping criterion is when the relative residual is less than $10^{-6}$, and DOF is the total degree of freedom.)} \label{tab:Poisson-hypre}\begin{tabular}{|cc|c|c|c|c|c|}
\cline{1-7}
Element Type & DOF & $\theta = 0.25$ & $\theta = 0.5$ & $\theta = 0.7$ & $\theta = 0.8$ & $\theta=0.9$
\\ 
\cline{1-7}
$P^{1,0}$ & 13M & 5 & 5 & 8 & 10 & 13
\\
\cline{1-7}
$P^{2,0}$ & 13M & 6 & 7 & 10 & 13 & 17
\\
\cline{1-7}
$P^{3,0}$ & 13M & 8 & 10 & 12 & 15 & 18
\\
\cline{1-7}
$P^{4,0}$ & 12M & $>$500 & $>$500 & 17 & 18 & 22
\\
\cline{1-7}  
\end{tabular} 
\end{table}

Studies have proposed using a two-level approach to handle Poisson equations on the higher-order finite element spaces. Such an approach would consist of (1) a smoother for the Poisson equation on the higher order finite element spaces, (2) transfer operators between the higher-order finite element spaces and the lower-order finite element spaces, and (3) an AMG method applied to the Poisson equations defined on the lower-order finite element space. Shu, Sun and Xu~\cite{Shu.S;Sun.D;Xu.J2006} has designed an algebraic multigrid method by constructing lower order finite element coefficient matrices algebraically with the help of characteristics of Lagrangian finite element spaces. Their study is restricted to the quadratic and cubic Lagrangian finite element discretizations in 2D. Another attempt to use low-order finite element space for preconditioning can be found in Heys, et al.~\cite{Heys.Manteuffel.ea2005}. 

The solution technique to discrete Poisson equations is itself of great interest. However, even more compelling are Poisson-based solution techniques that can be applied to constructing an efficient solver for more complicated problems~\cite{Xu.J1996a}. Over the last few decades, intensive research has been devoted to developing efficient linear solvers for almost all kinds of sparse linear systems in scientific and engineering computing. The main idea of efficient preconditioning is to transform a seemingly intractable problem to a (sequence of) problem(s) that can be approximated rapidly. One such mathematical technique is a general framework called Auxiliary Space Preconditioning or ASP~\cite{Xu.J1996a,Xu2010}. This method represents a large class of preconditioners that (1) by using auxiliary spaces transform a complicated system into a sequence of simpler systems, and (2) construct efficient preconditioners with efficient solvers for these simpler systems. Based on fast Poisson solvers and analytic insight into PDEs or PDE systems, efficient solvers can be developed using the auxiliary space preconditioning framework for various cases that arise in practical computations. Successful examples include simple and complex fluid problems, linear elasticity, and $H(\mathop{grad})$, $H(\mathop{div})$, and $H(\mathop{curl})$ systems with applications to the Maxwell equations~\cite{Hiptmair.Xu2007,Xu2010,Lee2011a}. 


In this paper, we revisit the algorithm proposed in~\cite{Xu.J1996a} for solving a large-scale discrete second-order elliptic equations by high-order finite element methods. Moreover, using easily available mesh information, we provide a parallel implementation of this auxiliary space preconditioner and analyze its performance for problems with about half a billion unknowns in terms of the robustness, efficiency, and scalability. This paper makes an additional contribution by providing an alternative proof for the convergence rate of the proposed algorithms. Lastly, the proposed method will be applied to solving the 3D Stokes equation on unstructured meshes. It is noteworthy  that this proposed preconditioner is user-friendly and can improve the robustness, efficiency, and scalability of the solution to the Stokes equation compared with pure AMG methods. This indicates that the proposed method can also make a useful building block for other Poisson-based solvers. 


Throughout this paper, we will use the following notation. The symbol $L^2_0$ denotes the space of all square integrable functions, $L^2$, whose entries have zero mean values. Let $H^k$ be the standard Sobolev space of the scalar function whose weak derivatives up to order $k$ are square integrable, and, let $\|\cdot\|_k$ and $|\cdot|_k$ denote the standard Sobolev norm and its corresponding seminorm on $H^k$, respectively. Furthermore, $\|\cdot\|_{k,\omega}$ and $|\cdot |_{k,\omega}$ denote the norm $\|\cdot\|_k$ and the semi-norm $|\cdot |_k$ restricted to the domain $\omega\subset\Omega$, respectively. We use the notation $X\lesssim(\gtrsim)Y$ to denote  the existence of a generic constant $C$, which depends only on $\Omega$, such that $X\leq (\geq) CY$.

The rest of the paper is organized as follows. In \S\ref{sec:algorithm}, using the auxiliary space preconditioning framework, we present the construction of the geometric hierarchy and a two-level method for the Poisson equation from high-order finite element discretizations. In \S\ref{sec:gamg} and \S\ref{sec:proof}, we analyze the convergence of the proposed two-level algorithm by casting it into the augmented matrix formulation by Griebel~\cite{Griebel.M1994}. In \S \ref{sec:mfem}, we discuss the preconditioning techniques for saddle point problems from the mixed finite element for the Stokes equation. In \S \ref{sec:numer}, a number of numerical experiments are reported and summarized to demonstrate the efficiency and robustness of our parallel implementation. 


\section{A geometric--algebraic multigrid algorithm}\label{sec:algorithm}

This section is devoted to present the algebraic multigrid methods for the Poisson equations discretized by the higher order finite element methods with geometric hierarchy between higher order finite elements and the lower order finite element spaces consisting of piecewise linear elements. 

Let $\Omega \subset \mathbb{R}^{d}$ be a bounded polyhedral domain. We consider the Poisson equation 
\begin{equation}\label{eqn:Poisson}
-\Delta u = f  \qquad \mbox{in } \Omega,
\end{equation}
subject to zero Dirichlet boundary condition $u = 0 \mbox{ on } \partial \Omega$.  We then consider the following weak formulation of \eqref{eqn:Poisson}: Find $u \in V = H^1_0(\Omega)$, such that
\begin{equation}\Label{main:eq}
a(u,v) = \langle f,v \rangle \qquad \forall\, v \in V,
\end{equation}
where 
\begin{equation}
a(u,v) := \int_\Omega \nabla u \cdot \nabla v \, dx \quad \mbox{and} \quad \langle f, v\rangle := \int_\Omega f\, v\, dx \qquad \forall\, u,v \in V. 
\end{equation}

We now discretize the equation (\ref{main:eq}) using the finite element method. To introduce the finite element spaces. Assume that $\mathcal{T}_h$ is a shape-regular triangular (tetrahedral) mesh of $\Omega$. For any $T\in\mathcal{T}_h$, let $P^k(T)$ be the set of polynomials on $T$ of degree less than or equal to $k$. We denote the piecewise continuous $P^k$ Lagrangian finite element space as $V_h:=P^{k,0}$. In this paper, $V_h$ denotes a finite element space consisting of the $k$-th ($k \geq 2$) order piecewise continuous polynomials, such as quadratic, cubic or quartic polynomials. That is to say
\begin{equation}
V_h := \{ v \in C(\Omega): v|_{T} \in P^k(T), \; \forall\, T \in \mathcal{T}_h \} = {\rm span}\{\phi_1,\ldots,\phi_{n_{_h}} \},
\end{equation}
where $n_{_h}$ is the total number of degrees of freedom and $\{\phi_i\}_{i=1,\ldots,n_h}$ are the standard $k$-th order Lagrange basis functions.
The discrete weak formulation of \eqref{main:eq} can be written as
\begin{equation}\Label{main:deq}
a(u_h,v_h) = \langle f,v_h \rangle \qquad \forall\, v_h \in V_h.
\end{equation}

We introduce an auxiliary space, the continuous piecewise linear polynomial space, 
\begin{equation}
V_H := \{ v \in C(\Omega): v|_{T} \in P^1(T), \; \forall\, T \in \mathcal{T}_h \} = {\rm span}\{\psi_1,\ldots,\psi_{n_{_H}}\}, 
\end{equation} 
where $\{\psi_j\}_{j=1,\ldots,n_H}$ are the canonical basis functions or the hat functions. We denote, by $\{x^h_i\}_{i=1, \ldots, n_{_h}}$ and $\{x^H_i\}_{i=1,\ldots,n_{_H}}$, the set of evenly-spaced nodes where the degree of freedom (DOF) for the Lagrange finite spaces $V_h$ and $V_H$ are defined, respectively. Figure~\ref{fig:tetrahedron} shows the local ordering of $x^h_i$ on a single simplex in 3D. 
\begin{figure}[h!!] 
\centering
\subfigure{
\includegraphics[width=0.26\linewidth] {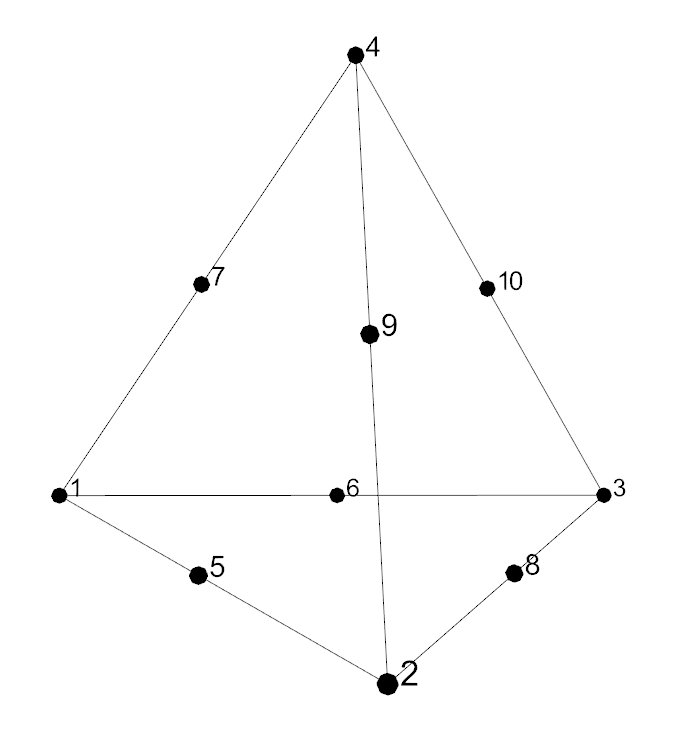}
}
\subfigure{
\includegraphics[width=0.31\linewidth] {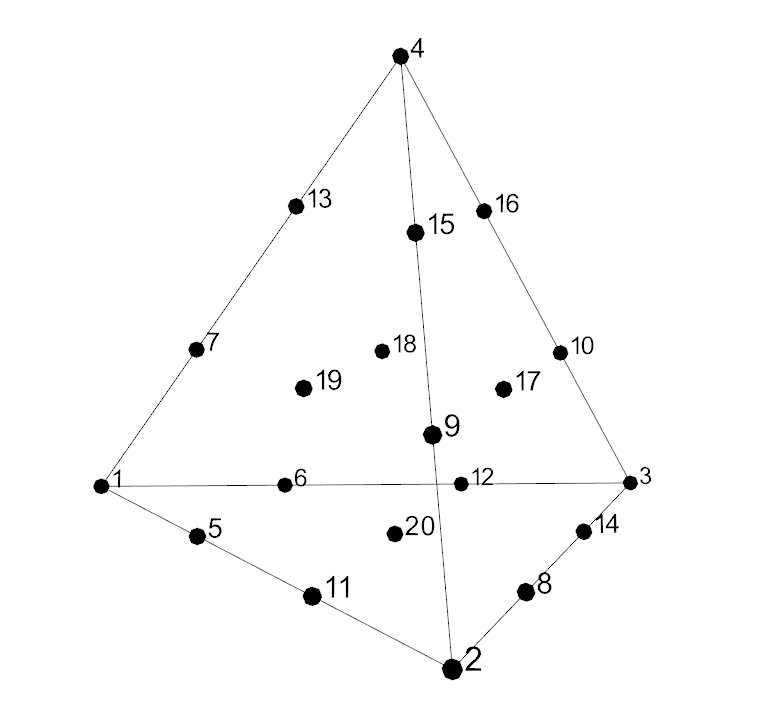}
}
\subfigure{
\includegraphics[width=0.31\linewidth] {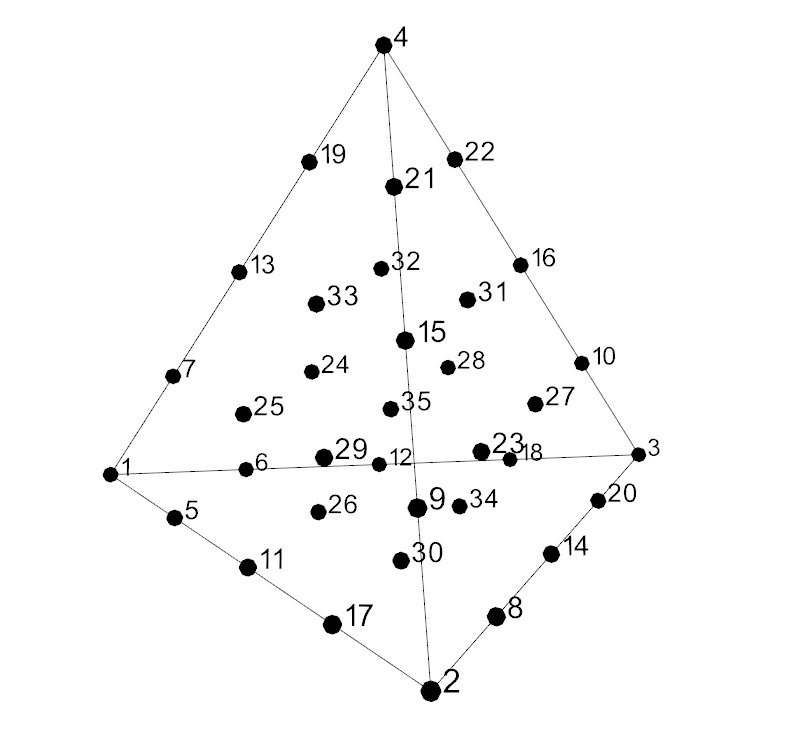}
}
\caption{Local numbering of nodes $x_i^h$'s in a single tetrahedron used for $P2$, $P3$ and $P4$ FEMs}\label{fig:tetrahedron}
\end{figure}

Throughout this paper, we use the following convention for the vector representation of a function in $V_h$: For any $v_h = \sum_{i=1}^{n_{_h}} v_i \phi_i \in V_h$ and $w_H = \sum_{i=1}^{n_{_H}} w_i \psi_i \in V_H$, we denote by ${\bf{v}}_h$ and ${\bf{w}}_{_H}$, the vector representation of $v_h$ and $w_h$, respectively. Namely, ${\bf{v}}_{_h} = (v_1,\ldots,v_{n_{_h}})^T$ and ${\bf{w}}_{_H} = (w_1,\ldots,w_{n_{_H}})^T$. Similarly, the symbol ${\bf{f}}_{_h}$ denotes $(f_1,\ldots,f_{n_{_h}})^T$ with $f_i = \langle f, \phi_i \rangle$ for all $i=1,\ldots,n_{_h}$. The equation (\ref{main:deq}) can then be cast into the following equivalent finite-dimensional linear system of equations: Find ${\bf{u}}_{_h}$ such that
\begin{equation}\Label{mat:eq}
\mathsf{A}_{h} {\bf{u}}_{_h} = {\bf{f}}_{_h},
\end{equation}
where $\mathsf{A}_{h}$ is a symmetric positive definite matrix with $(\mathsf{A}_{h})_{ij}=a(\phi_j, \phi_i)$.

We are now at the position to present a geometry-based algebraic multigrid (GAMG) method by using the auxiliary linear finite element space. For this purpose, we start by introducing the transfer operators between $V_h$ and $V_H$. 
Since $k \geq 2$, any basis function $\psi_j \in V_H$ can be represented by the basis functions $\{\phi_i\}_{1\leq i\leq n_{_h}}$ of the space $V_h$. Namely, for any $j =1, \ldots n_{_H}$, there exists ${\bf{c}}_j = (c_j^1,\ldots,c_j^{n_{_h}})^T \in \mathbb{R}^{n_{_h}}$ such that
\begin{equation}
\psi_j = \sum_{i=1}^{n_{_h}} c^i_j\phi_i.
\end{equation}
These coefficients $c_j^i$'s can be easily obtained by $c_j^i = \psi_j(x_i^h)$ for any $i=1,\ldots,n_{_h}$ and $j=1,\ldots,n_{_H}$. Hence, we have that
\begin{equation}
\psi_j = \sum_{i \in N_h^H(j)} \psi_j(x_i^h) \phi_i,
\end{equation}
where $N_h^H(j) = \{ i \in \{1,\ldots,n_{_h}\} : {\rm supp}(\phi_i) \cap {\rm supp} (\psi_j) \neq \emptyset \}$. We define the transfer operators
\begin{equation}\Label{Prol}
\mathsf{I}_{\mathsf{P}} = ({\bf{c}}_1,\ldots,{\bf{c}}_{n_{H}}) \in \mathbb{R}^{n_{_h} \times n_{_H}} \quad \mbox{ and } \quad \mathsf{I}_{\mathsf{R}} = \mathsf{I}_{\mathsf{P}}^T. 
\end{equation}
We note that the transfer operators $\mathsf{I}_{\mathsf{R}}$ and $\mathsf{I}_{\mathsf{P}}$ can be generated easily with the mesh $\mathcal{T}_h$ available. An alternate (algebraic) approach for generating them has been introduced by Shu, Sun, and Xu~\cite{Shu.S;Sun.D;Xu.J2006}.

Now, let $\mathsf{A}_{_h}$ and $\mathsf{A}_{_H}$ be the stiffness matrices defined by \eqref{main:deq} on the finite element spaces $V_h$ and $V_H$, respectively. Then we have the Galerkin relation $\mathsf{A}_{_H} = \mathsf{I}_{\mathsf{R}} \mathsf{A}_{_h} \mathsf{I}_{\mathsf{P}}$. Let $\mathsf{G}_h$ be a smoother and we can state the following two-level algorithm: 
\begin{algorithm}[A Two-level Method]\label{gamg}
Given an initial iterate, ${\bf{u}}_{_h}^0$ on the fine grid, we perform the following steps until convergence for $\ell=0,1,\ldots$


{\bf{Step 1.}} {\it{Solve the coarse grid equation}} 
$$
\mathsf{A}_{H} {\bf{w}}_{_H} = \mathsf{I}_{\mathsf{R}} ({\bf{f}}_{_h} - \mathsf{A}_{h} {\bf{u}}_{_h}^\ell)
$$

{\bf{Step 2.}} {\it{Correction}} 
$$
\widetilde{{\bf{u}}}_{_h} = {{\bf{u}}}_{_h}^\ell + \mathsf{I}_{\mathsf{P}} \, {\bf{w}}_{_H} 
$$

{\bf{Step 3.}} {\it{Postsmoothing}}
$$
{\bf{u}}_{_h}^{\ell+1} = \widetilde{{\bf{u}}}_{_h} + \mathsf{G}_h ({\bf{f}}_{_h} - \mathsf{A}_{h}\widetilde{{\bf{u}}}_{_h})
$$
\end{algorithm}

\begin{remark}[The Two-Level Method and GAMG]\label{rem:2-1vsGAMG}\rm
In Step 1 of Algorithm~\ref{gamg}, we can employ an AMG method or a few AMG cycles to solve the coarse-level problem approximately. We can also add a presmoothing step in front of Step 1 to make the method symmetric. This particular version of Algorithm~\ref{gamg} will then be referred to as the geometric-algebraic multigrid (GAMG) method as easily available geometric information (mesh) is used in our implementation of Step 3. 
\end{remark}


\section{Augmented matrix formulation of the two-level method}\label{sec:gamg}

In this section, we introduce an equivalent form of Algorithm~\ref{gamg} in terms of the augmented algebraic systems based on so-called the redundant representation of the solution space. This observation was originally made by Griebel \cite{Griebel.M1994}, however, the convergence rate estimate in this framework has not been seen in literature. The main idea lies in the following redundant representation of the functions in the space $V_h$: For any $v_{_h} \in V_h$, we have the representation
\begin{equation}\Label{Rep:Aug}
v_{_h} = \sum_{i=1}^{n_{_H}} \bar{v}_i \psi_i + \sum_{i=1}^{n_{_h}} \widehat{v}_i \phi_i.
\end{equation}

We notice that the above representation is not unique since the basis functions used in the representation are not independent, which is why \eqref{Rep:Aug} is also called the redundant representation of functions in $V_h$. Based on \eqref{Rep:Aug}, we can consider the discrete weak formulation: Let $\widetilde{V}_h = {\rm span} \{\psi_1,\ldots,\psi_{n_{_H}}, \phi_1,\ldots, \phi_{n_{_h}}\}$ and find $v_{_h} = \sum_{i=1}^{n_{_h}} \widehat{v}_i \phi_i + \sum_{i=1}^{n_{_H}} \bar{v}_i \psi_i \in \widetilde{V}_h$,  such that 
\begin{equation}\label{weakaug}
a(v_{_h}, w_h) = \langle f, w_h \rangle \qquad \forall\, w_h \in \widetilde{V}_h. 
\end{equation}
Let $\widehat{{{v}}} = (\widehat{v}_1,\ldots,\widehat{v}_{n_{_h}})^T$ and 
$\overline{{{v}}} = (\bar{v}_1,\ldots,\bar{v}_{n_{_H}})^T$. It is then easy to establish that the resulting system of equations from the aforementioned weak formulation (\ref{weakaug}) leads to the following augmented matrix systems~\cite{Griebel.M1994}:
\begin{Lemma}\Label{Augmented}
Let $\mathsf{I}_{\mathsf{P}}$ and $\mathsf{I}_{\mathsf{R}}$ be the prolongation and restriction given in (\ref{Prol}), respectively. Then the problem (\ref{main:deq}) can be written as the following matrix equation
\begin{equation}\Label{Aud:Mented}
\mathcal{A} \left( \begin{array}{c} \overline{{{v}}} \\ \widehat{{{v}}} \end{array} \right )= \widetilde{f} \qquad \text{i.e.} \;\;
\left ( \begin{array}{cc}
\mathsf{I}_{\mathsf{R}} \mathsf{A}_h \mathsf{I}_{\mathsf{P}}
 & \mathsf{I}_{\mathsf{R}} \mathsf{A}_h \\
\mathsf{A}_h \mathsf{I}_{\mathsf{P}}
 & \mathsf{A}_h \end{array} \right )  \left( \begin{array}{c} \overline{{v}} \\ \widehat{{{v}}} \end{array} \right ) =  \left( \begin{array}{c} \mathsf{I}_{\mathsf{R}} \, {\bf{f}}_h \\ {\bf{f}}_h \end{array} \right ).
\end{equation}
\end{Lemma}
\begin{proof}
From the following relation that
\begin{equation}
\psi_i = \sum_{j=1}^{n_{_h}} c^j_i\phi_j,  \qquad i =1,\ldots,n_{_H},
\end{equation}
we can deduce that
\begin{equation}
a(\psi_i,\psi_j) = (\mathsf{I}_{\mathsf{R}} \mathsf{A}_h \mathsf{I}_{\mathsf{P}}
)_{ij}
\end{equation}
and
\begin{equation}\displaystyle
a(\psi_i,\phi_j) = a \Big ( \sum_{k=1}^{n_{_h}} c_i^k \phi_k, \, \phi_j \Big ) =
\sum_{k=1}^{n_{_h}} c_i^k a(\phi_k,\phi_j) = (\mathsf{I}_{\mathsf{R}} \mathsf{A}_h)_{ij}.
\end{equation}
We also note that 
$$
\langle f, \psi_i \rangle = \Big \langle f, \, \sum_{k=1}^{n_{_h}} c_j^k \phi_k \Big \rangle = \sum_{k=1}^{n_{_h}} c_j^k \langle f, \phi_k \rangle = \mathsf{I}_{\mathsf{R}} {\bf{f}}_h. 
$$
This completes the proof.
\end{proof}

There are a few interesting properties of the augmented matrix $\mathcal{A}$ given in (\ref{Aud:Mented}) which will be useful in later sections. First of all, the matrix $\mathcal{A}$ is singular with positive diagonal entries. Moreover, the range and the null space of $\mathcal{A}$ denoted by $\R(\mathcal{A}) =\K^\perp $ and $\K = \K(\mathcal{A})$, respectively, can be  characterized as follows:
\begin{eqnarray*}
\R(\mathcal{A}) = \left \{ \left ( \begin{array}{c} \mathsf{I}_{\mathsf{R}} {\bf{v}}_{_h} \\
{\bf{v}}_{_h}
\end{array} \right ) : v_{_h} \in V_h \right \} \,\, \mbox{ and } \,\,
\K(\mathcal{A})  = \left \{ \left ( \begin{array}{c} {\bf{c}}_{_H} \\ - \mathsf{I}_{\mathsf{P}}
 {\bf{c}}_{_H}
\end{array} \right ) : c_{_H} \in V_H \right \}.
\end{eqnarray*}

Let $\widetilde{{\bf{e}}}_k$ for $k=1,\ldots,n_{_H} + n_{_h}$ be the canonical basis for $\mathbb{R}^{n_{_H}+n_{_h}}$. We denote the solution space of the augmented matrix system \eqref{Aud:Mented} as 
\begin{equation}
\mathcal{V} := \left \{ \widetilde{v} = \left ( \begin{array}{c} \overline{{{v}}}  \\ \widehat{{{v}}} \end{array} \right ) \, : \,  v_{_h} = \sum_{i=1}^{n_{_H}} \bar{v}^i \psi_i + \sum_{i=1}^{n_{_h}} \hat{v}^i \phi_i, \quad \forall\, \,v_{_h} \in \widetilde{V}_h \right \}
\subseteq \mathbb{R}^{n_{_H}+n_{_h}}.
\end{equation}
It is worthy to note that Algorithm~\ref{gamg}\footnote{We can also analyze the  Algorithm~\ref{gamg} with presmoothing by modifying the space decomposition slightly to make it symmetric.} can analyzed in the framework of Successive Subspace Corrections (SSC)~\cite{Xu.J1992a} with the subspace decomposition
$$
\mathcal{V} = \mathcal{V}_0 + \mathcal{V}_1 + \cdots + \mathcal{V}_{n_{_h}}, 
$$ 
where $\mathcal{V}_0 = {\rm span}\,\{\widetilde{{\bf{e}}}_1,\ldots,\widetilde{{\bf{e}}}_{n_{_H}} \}$ and $\mathcal{V}_j = {\rm span} \{ \widetilde{{\bf{e}}}_{n_{_H}+j} \}$ for $j = 1, \ldots, n_{_h}$. In this setting, the SSC method can be written as follows: 
 
\begin{algorithm}[Successive Subspace Correction Method]\Label{algo}
Let $\widetilde{{u}}^0 \in \mathcal{V}$ be given.

\hskip 0.5cm $\mbox{{\bf{for}} } \ell = 1,2,\ldots $

\hskip 1.1cm $\widetilde{{u}}_0^{\ell-1} = \widetilde{{u}}^{\ell-1}$

\hskip 1.1cm $\mbox{{\bf{for }}} k=0,1,2,\ldots,n_{_h}$

\hskip 1.7cm $\mbox{Find } \widetilde{{w}}_{k} \in \mathcal{V}_{k}: 
\;\; \left ( \mathcal{A} \widetilde{{w}}_{k}, \widetilde{{v}}_{k} \right ) = \left ( \widetilde{{f}}, \widetilde{{v}}_{k} \right ) - \left ( \mathcal{A} \widetilde{{u}}^{\ell-1}_{k-1}, \widetilde{{v}}_{k} \right) \quad \forall\, \widetilde{{v}}_{k} \in \mathcal{V}_{k}$

\hskip 1.7cm $\widetilde{{u}}_{k}^{\ell-1} = \widetilde{{u}}_{k-1}^{\ell-1} + \widetilde{{w}}_{k} $

\hskip 1.1cm $\mbox{{\bf{endfor}}}$

\hskip 1.1cm $\widetilde{{u}}^\ell = \widetilde{{u}}_{n_{_h}}^{\ell-1}$

\hskip 0.5cm $\mbox{{\bf{endfor}}} $
\end{algorithm}

The error transfer operator of the above algorithm can be identified as\begin{equation}\Label{ErrorSSC}
\mathcal{E} = (\mathcal{I} - \mathcal{P}_{n_{_h}}) (\mathcal{I} - \mathcal{P}_{n_{_h}-1}) \cdots (\mathcal{I} - \mathcal{P}_{0}),
\end{equation}
where $\mathcal{I} : \mathcal{V} \mapsto \mathcal{V}$ is the identity matrix and
$\mathcal{P}_j : \mathcal{V} \mapsto \mathcal{V}_{j}$ for $j=0,\ldots,n_{_h}$ is the $\mathcal{A}$-projection onto the space $\mathcal{V}_{j}$. More precisely, $\mathcal{P}_0 : \mathcal{V} \mapsto \mathcal{V}_H$ can be defined as
\begin{equation*}
\mathcal{P}_0 \widetilde{v} = \left ( \begin{array}{c} 
\overline{v} + (\mathsf{I}_{\mathsf{R}} \mathsf{A}_h  \mathsf{I}_{\mathsf{P}})^{-1} \mathsf{I}_{\mathsf{R}} \mathsf{A}_h \widehat{v} 
\\ 0 \end{array} \right ) = \left ( \begin{array}{c} 
\overline{v} + \mathsf{P}_0 \widehat{v}  \\ 0 \end{array} \right )
\qquad \forall\, \widetilde{v} = 
\left(
\begin{array}{c}
\overline{v} \\
\widehat{v}
\end{array}
\right) \in \mathcal{V}, 
\end{equation*}
where $\mathsf{P}_0 = (\mathsf{I}_{\mathsf{R}} \mathsf{A}_h \mathsf{I}_{\mathsf{P}})^{-1} \mathsf{I}_{\mathsf{R}} \mathsf{A}_h$. For $j = 1, \ldots, n_{_h}$, we define the projections
\begin{equation*}
\mathcal{P}_j \widetilde{v} := \frac{(\A \widetilde{{\bf{e}}}_{n_{_H}+j}, \widetilde{v} )}{( \A \widetilde{{\bf{e}}}_{n_{_H}+ j}, \widetilde{{\bf{e}}}_{n_{_H}+ j})} \widetilde{{\bf{e}}}_{j} =  \frac{(\mathsf{A}_h {\bf{e}}_{j}, (\mathsf{I}_{\mathsf{P}} \overline{v} + \widehat{v}))}{(\mathsf{A}_h {\bf{e}}_{j}, {\bf{e}}_{j})} \widetilde{{\bf{e}}}_{n_{_H}+ j} = 
(\mathsf{P}_{j} (\mathsf{I}_{\mathsf{P}} \overline{v} + \widehat{v}), {\bf{e}}_{j})  \widetilde{{\bf{e}}}_{n_{_H}+j},  
\end{equation*} 
where ${\bf{e}}_j$ for $j=1,\ldots,n_{_h}$ is the canonical basis for $\mathbb{R}^{n_{_h}}$ and $\mathsf{P}_j$ is $\mathsf{A}_h$-projection onto the space ${\rm span}\{{\bf{e}}_j\}$ given by 
$$
\mathsf{P}_j\widehat{v} = \frac{(\mathsf{A}_h {\bf{e}}_j, \widehat{v})}{(\mathsf{A}_h {\bf{e}}_j, {\bf{e}}_j)} {\bf{e}}_j \qquad \forall\, \widehat{v} \in \mathbb{R}^{n_{_h}}. 
$$ 

\begin{remark}[Algorithm~\ref{gamg} and Block Gauss-Seidel Method]\label{rem:GS}\rm
We apply the space decomposition
\begin{equation}
\mathcal{V} = \mathcal{V}_{_H} + \mathcal{V}_{_h},
\end{equation}
where
\begin{equation*}
\mathcal{V}_{_H} = {\rm span}\,\{\widetilde{{\bf{e}}}_1,\ldots,\widetilde{{\bf{e}}}_{n_{_H}} \} \quad \mbox{ and } \quad 
\mathcal{V}_{_h} =  {\rm span}\,\{\widetilde{{\bf{e}}}_{n_{_H}+1},\ldots, \widetilde{{\bf{e}}}_{n_{_H} + n_{_h}} \}.
\end{equation*}
We decompose of the matrix $\mathsf{A}_h$ in the following form  
\begin{equation}
\mathsf{A}_h = \mathsf{D} - \mathsf{L} - \mathsf{L}^T,
\end{equation} 
where $\mathsf{D} = (a_{ii})_{i=1,\ldots,n_{_h}}$ is the diagonal part of $\mathsf{A}_h$ and $\mathsf{L} = (\ell_{ij})_{i,j=1,\ldots,n_{_h}}$ with $\ell_{ij} = 0$ for $i > j$ and $\ell_{ij} = -a(\phi_j,\phi_i)$ for $i < j$, i.e., the strictly lower triangular part of $\mathsf{A}_h$. Similarly we can decompose $\mathcal{A}$ as follows 
\begin{equation}
\mathcal{A} = \mathcal{D} - \mathcal{L} - \mathcal{L}^T, 
\end{equation}
where 
\begin{equation*}
\mathcal{D} :=  \left ( \begin{array}{cc}
\mathsf{I}_{\mathsf{R}} \mathsf{A}_h \mathsf{I}_{\mathsf{P}}
 & 0 \\
0 & \mathsf{D} \end{array} \right ) 
\quad \mbox{ and } \quad 
\mathcal{L} := \left ( \begin{array}{cc}
0 & 0 \\
- \mathsf{A}_h \mathsf{I}_{\mathsf{P}}
 & \mathsf{L} \end{array} \right ).
\end{equation*}
In fact, we can easily show that the two-level method Algorithm~\ref{gamg} is equivalent to the Gauss-Siedel method for the augmented system of equations~\eqref{Aud:Mented}:
\begin{equation}\label{bgamg}
\widetilde{{v}}^{\ell+1} = \widetilde{{v}}^\ell + (\mathcal{D} - \mathcal{L})^{-1} \, ( \, \widetilde{{f}} - \mathcal{A} \widetilde{{v}}^\ell \, ) \qquad \ell=0,1,\ldots 
\end{equation}
\end{remark}

\section{Convergence rate estimate for the two-level method} \label{sec:proof}

In this section, we establish a convergence rate estimate for Algorithm~\ref{gamg} using the formulation introduced in the previous section. We denote a semi inner product $(\cdot,\cdot)_\A : \mathcal{V} \times \mathcal{V} \mapsto \mathbb{R}$ and the induced semi-norm by $|\cdot|_\A = (\cdot,\cdot)_\A^{1/2}$. We now establish a convergence rate identity for the error transfer operator~\eqref{ErrorSSC}. Note that  relevant estimates have been reported in \cite{Lee.Y;Wu.J;Xu.J2002a,lee;wu;xu;zikatanov2008c}, but we provide a proof for completeness.

\begin{Theorem}[Convergence Rate Identity]\Label{ssc:cr}
The convergence rate for the iterative method (\ref{bgamg}) can be given by
\begin{equation}
|\mathcal{E}|_{\A}^2 = 1 - \frac{1}{K},
\end{equation} 
where  
\begin{eqnarray*}
K =1 + \sup_{\widetilde{u} \in \K^\perp} \inf_{\widetilde{c} \in \K} 
\frac{(\mathcal{S} (\widetilde{u}+\widetilde{c}), (\widetilde{u}+\widetilde{c}))}{(\widetilde{u},\widetilde{u})_\A}  = 
\sup_{\widetilde{u} \in \K^\perp}
\inf_{\widetilde{c} \in \K} \frac{\sum_{i=0}^{n_{_h}} 
\left | \mathcal{P}_i \left ( \sum_{j=i}^{n_{_h}} \widetilde{u}_j \right ) \right |_{{\A}}^2}{ (\widetilde{u},\widetilde{u})_\A },
\end{eqnarray*}
where $\mathcal{S} = {\mathcal{L}} {\mathcal{D}}^{-1} {\mathcal{L}}^T$ and $\widetilde{u}_j \in \mathcal{V}_j$ for $j=0,\ldots,n_{_h}$.
\end{Theorem}
\begin{proof}
Let $\mathcal{B} := (\mathcal{D} - \mathcal{L})^{-1}$. From Remark~\ref{rem:GS}, we have $\mathcal{E}=\mathcal{I}-\mathcal{B}\mathcal{A}$. By the definition of the semi-norm $|\mathcal{E}|_{\A}$, we have the following identity
\begin{eqnarray*}
|\mathcal{E}|_{\A}^2 &=& 
\sup_{\widetilde{u} \in \K^\perp } \frac{(\mathcal{E} \widetilde{u}, \mathcal{E} \widetilde{u})_{\A}}{(\widetilde{u},\widetilde{u})_{\A}} 
= \sup_{\widetilde{u} \in \K^\perp} \frac{\left ( ({\I} - {\B}{\A}) \widetilde{u},
({\I} - {\B} {\A}) \widetilde{u} \right )_{\A}}{(\widetilde{u},\widetilde{u})_{\A}} \\
&=& \sup_{\widetilde{u} \in \K^\perp} \frac{ \left( ( {\I} - {\B} {\A} )^*( {\I} - {\B} {\A} ) \widetilde{u}, \widetilde{u} \right )_{\A} }{(\widetilde{u},\widetilde{u})_{\A}},
\end{eqnarray*}
where $( {\I}  - {\B} {\A})^*$ is the adjoint operator of ${\I} - {\B} {\A}$ with respect to the semi inner product $(\cdot,\cdot)_{{\A}}$. 

We notice that $( {\I}  - {\B} {\A})^* = {\I} - {\B}^T {\A}$, where $\B^T$ is the adjoint of $\B$ with respect to the usual $\ell^2$ inner product. Furthermore, since
\begin{eqnarray*}
({\I} - {\B} {\A})^*({\I} - {\B} {\A}) &=& \I - \B^T (\B^{-T} + \B^{-1} - \A) \B \A  \\ 
&=&  \I - (\B^{-1} \mathcal{D}^{-1} \B^{-T})^{-1} \A  = \I - (\A + \mathcal{S})^{-1} \A,
\end{eqnarray*}
we then have that
\begin{eqnarray*}
|\E|_{\A}^2 = 1 - \inf_{\widetilde{u} \in \K^\perp} \frac{((\A + {\mathcal{S}})^{-1} {\A}
\widetilde{u},\widetilde{u})_{\A}}{(\widetilde{u},\widetilde{u})_{\A}}.
\end{eqnarray*}
We shall now define ${\mathcal{M}} = {\A}^{1/2}( {\A} + {\mathcal{S}} )^{-1} {\A} {\A}^{1/2}$ and obtain the following identity : 
\begin{eqnarray*}
K &=& \sup_{\widetilde{u} \in \K^\perp}
\frac{(\widetilde{u}, \widetilde{u})_{\A}}{((\A+ \mathcal{S})^{-1} \A \widetilde{u}, \widetilde{u})_{\A}} = \sup_{\widetilde{u} \in \K^\perp} \frac{(\widetilde{u}, \widetilde{u})_{\A}}{(\A^{-1/2} {\mathcal{M}} \A^{-1/2}\widetilde{u}, \widetilde{u})_{\A}} 
\\
&=& \sup_{\widetilde{u} \in \K^\perp} \frac{(\A^{-1/2} {\mathcal{M}} \A^{-1/2} \widetilde{u}, \widetilde{u})_{\A}}{( \A^{-1/2} {\mathcal{M}} {\A}^{-1/2} \widetilde{u}, {\A}^{-1/2} {\mathcal{M}} {\A}^{-1/2} \widetilde{u})_{{\A}}}.
\end{eqnarray*}
The last equality is from the fact that $\mathcal{M} : \K^\perp \mapsto \K^\perp$ is symmetric and positive definite matrix and by replacing $\widetilde{u}$ by $\mathcal{M}^{1/2}\widetilde{u}$. Let $Q : \mathcal{V} \mapsto \K^\perp$ be the $\ell^2$-orthogonal projection, for which
\begin{equation}
{\A} Q \widetilde{v} = {\A} \widetilde{v} \qquad \forall\, \widetilde{v} \in \mathcal{V}. 
\end{equation}
We now write $(\A + {\mathcal{S}})^{-1} {\A}\widetilde{u} = \widetilde{w} + \widetilde{c}(\widetilde{w})$, where $\widetilde{w} := Q(\A + {\mathcal{S}})^{-1} \A \widetilde{u} \in \K^\perp$ and $\widetilde{c}(\widetilde{w}) \in \K$. Note that $\widetilde{c}(\widetilde{w})$ is
uniquely determined by $\widetilde{w}$. Therefore, due to the fact that 
$$
\A^{-1/2} {\mathcal{M}} {\A}^{-1/2} \widetilde{u} = (\A + \mathcal{S})^{-1}\A \widetilde{u} = \widetilde{w} + \widetilde{c}(\widetilde{w}) \quad \mbox{ and } \quad (\A + \mathcal{S})(\widetilde{w} + \widetilde{c}(\widetilde{w})) = \A \widetilde{u},$$  
we immediately obtain that
\[
K = \sup_{\widetilde{w} \in \K^\perp } \frac{ \left ( (\widetilde{w} + \widetilde{c}(\widetilde{w})), \, ( \A + {\mathcal{S}})(\widetilde{w} + \widetilde{c}(\widetilde{w}) \right )}{(\widetilde{w}, \widetilde{w})_{{\A}}} = \sup_{\widetilde{w} \in \K^\perp } \inf_{\widetilde{c} \in \K} \frac{(({\A} + {\mathcal{S}})(\widetilde{w} + \widetilde{c}),(\widetilde{w} + \widetilde{c}))}{(\widetilde{w}, \widetilde{w})_{\A}}. 
\]

The last equality is obtained by the following reasoning: For a fixed $\widetilde{w} \in \K^\perp$, we assume that 
\begin{equation}
\widetilde{\xi}  = {\rm arg} \inf_{\widetilde{c} \in \K } \frac{\big( (\A + {\mathcal{S}})(\widetilde{w} + \widetilde{c}),(\widetilde{w} + \widetilde{c}) \big)}{(\widetilde{w} ,\widetilde{w})_{\A}}.
\end{equation}
Then the first order optimality condition implies that $\xi$ satisfies
\begin{equation}
\left( ({\A} + {\mathcal{S}})(\widetilde{w} + \widetilde{\xi}),\, \widetilde{c} \right) = 0 \qquad \forall\,  \widetilde{c} \in \K.
\end{equation}
This in turn implies that $({\A}+ {\mathcal{S}})(\widetilde{w} + \widetilde{\xi}) \in \K^\perp$ and $\widetilde{\xi} = \widetilde{c}(\widetilde{w})$. This completes the proof.
\end{proof}


To establish a uniform convergence rate of the proposed two-level method (Algorithm~\ref{gamg}), we show that $K$ in Theorem \ref{ssc:cr} can be bounded by a generic constant independent of mesh size. The standard prolongation $I_h^H : V_H \mapsto V_h$ (the inclusion operator) and the restriction operator $I_H^h : V_h \mapsto V_H$ can be written as 
\begin{equation}
I_h^H w_{_H} = (\phi_1,\ldots,\phi_{n_{_h}}) \mathsf{I}_\mathsf{P} {\bf{w}}_H \quad \mbox{ and } \quad I_H^h v_{_h} = (\psi_1,\ldots,\psi_{n_{_H}}) \mathsf{I}_{\mathsf{R}} {\bf{v}}_{_h}.
\end{equation}
We introduce two additional restriction operators: the usual $L^2$ projection $Q_H : V_h \mapsto V_H$ and the elliptic projection $P_H : V_h \mapsto V_H$ defined by 
\begin{equation*}
(Q_Hv,w_{_H}) = (v,w_{_H}) \; \mbox{ and } \; a(P_H v, w_{_H}) = a(v, w_{_H}) \quad \forall v \in V_h, \; w_{_H} \in V_H,
\end{equation*}
respectively. It is clear that 
$
(\textsf{A}_h \textsf{I}_{\textsf{P}}{\bf{v}}_{_H}, \, \textsf{I}_{\textsf{P}}{\bf{v}}_{_H}) = \|v_{_H}\|^2_1
$
and the vector representation of $P_H v$ is simply $\textsf{A}_H^{-1} \textsf{I}_\textsf{R} \textsf{A}_h {\bf{v}}$.
Therefore, we have that, for any $v \in V_h$, the following inequality holds:
\begin{equation}\label{trivin}
(\textsf{A}_h \textsf{I}_{\textsf{P}} \textsf{A}_H^{-1} \textsf{I}_\textsf{R} \textsf{A}_h {\bf{v}}, \, \textsf{I}_{\textsf{P}}\textsf{A}_H^{-1} \textsf{I}_\textsf{R} \textsf{A}_h {\bf{v}}) = \|P_Hv\|_1^2 \leq \|v\|^2_1.  
\end{equation}

We are now ready to prove the main theoretical result in this paper. Note that there are abundant literatures regarding the uniform convergence of two grid methods. Unlike those reported in literatures such as~\cite{Shu.S;Sun.D;Xu.J2006}, our convergence estimate uses a different and novel technique, which is based on a convergence rate estimate for a singular system from the multilevel and redundant decomposition of the solution space. 
\begin{Theorem}[Uniform Convergence]
The following estimate holds true
\begin{eqnarray*}
K &=& 1 + \sup_{\widetilde{u} \in \K^\perp} \inf_{\widetilde{c} \in \K} 
\frac{(\mathcal{S} (\widetilde{u}+\widetilde{c}), (\widetilde{u}+\widetilde{c}))}{(\widetilde{u},\widetilde{u})_\A}  \\
&=& \sup_{\widetilde{u} \in \K^\perp} \inf_{\widetilde{c} \in \K}
\frac{\sum_{i=0}^{n_{_h}}  
\left | \mathcal{P}_i \left ( \sum_{j=i}^{n_{_h}} \widetilde{u}_j \right ) \right |_{{\A}}^2}{(\widetilde{u},\widetilde{u})_{\A}} \lesssim \sup_{v_{_h} \in V_h} \frac{\|v_{_h} - Q_H v_{_h}\|_1}{\|v_{_h}\|_1} + 1 \lesssim 1,  
\end{eqnarray*}
where $\widetilde{u}_j \in \mathcal{V}_j$ for $j=0,1,\ldots,n_{_h}$ such that $\sum_{j=0}^{n_{_h}} \widetilde{u}_j = \widetilde{u} + \widetilde{c}$. 
\end{Theorem}
\begin{proof}
Let $\widetilde{u} \in \K^\perp$ and $\widetilde{c} \in \K$ be given by 
\begin{equation}
\widetilde{u}  = \left ( \begin{array}{c} \mathsf{I}_{\mathsf{R}} {\bf{v}}_{_h} \\ {\bf{v}}_{_h} \end{array} \right ) \, \mbox{ and } \, \widetilde{c} = \left ( \begin{array}{c}
{\bf{v}}_{_H} \\ - \mathsf{I}_{\mathsf{P}} {\bf{v}}_{_H}
\end{array} \right ), \, \mbox{ where } {\bf{v}}_{_h} \in \mathbb{R}^{n_{_h}}, \,\, {\bf{v}}_{_H} \in
\mathbb{R}^{n_{_H}}.
\end{equation}
From Theorem~\ref{ssc:cr}, we have the identity 
\begin{eqnarray*}
K =\sup_{\widetilde{u} \in \K^\perp}
\inf_{\widetilde{c} \in \K} \frac{\sum_{i=0}^{n_{_h}} 
\left | \mathcal{P}_i \left ( \sum_{j=i}^{n_{_h}} \widetilde{u}_j \right ) \right |_{{\A}}^2}{(\widetilde{u},\widetilde{u})_{\mathcal{A}}}. 
\end{eqnarray*}
We note that the following identity holds true 
\begin{equation}
\left (\widetilde{u}, \widetilde{u} \right )_{\A} = \left \| \left(\mathsf{I} + \mathsf{I}_{\mathsf{P}}
\mathsf{I}_{\mathsf{R}} \right ) {\bf{v}}_{_h} \right \|_{\mathsf{A}_{_h}}^2 = \|v_{_h} + I_H^h v_{_h} \|_1^2, 
\end{equation}
where  ${\bf{v}}_{_h}$ is the vector representation of $v_{_h} \in V_h$ in terms of the basis functions $\{\phi_i\}_{i=1}^{n_{_h}}$. Therefore, by choosing ${\bf{v}}_{_H}$ as the vector representation of $v_{_H} = Q_H v_{_h}$, we obtain the following relation 
$$
\sum_{j=0}^{n_{_h}} \widetilde{u}_j =  \left ( \begin{array}{c} \mathsf{I}_{\mathsf{R}} {\bf{v}}_{_h} + {\bf{v}}_{_H} 
\\ {\bf{v}}_{_h} -  \mathsf{I}_{\mathsf{P}} {\bf{v}}_{_H} \end{array} \right ) = \left ( \begin{array}{c}  \mathsf{I}_{\mathsf{R}} {\bf{v}}_{_h} + {\bf{v}}_{_H}  
\\ {\bf{v}}_{_h} - \mathsf{I}_{\mathsf{P}} {\bf{v}}_{_H} \end{array} \right )
$$ 
and, in turn, 
\begin{eqnarray*}
\sum_{i=0}^{n_{_h}} \left | \mathcal{P}_i
\left ( \sum_{j=i}^{n_{_h}} \widetilde{u}_j \right ) \right |_{\A}^2 &=& \sum_{i=1}^{n_{_h}} \left | \mathcal{P}_i
\left ( \sum_{j=i}^{n_{_h}} \widetilde{u}_j \right ) \right |_{\A}^2 + \left |\mathcal{P}_0 \left ( \begin{array}{c}  \mathsf{I}_\mathsf{R} {\bf{v}}_{_h} + {\bf{v}}_{_H} \\ {\bf{v}}_{_h} - \mathsf{I}_{\mathsf{P}} {\bf{v}}_{_H} \end{array} \right )\right |_{\A}^2 \\
&=& \sum_{i=1}^{n_{_h}} \left \| \mathsf{P}_i \left ( \sum_{j = i}^{n_{_h}} \widehat{u}_j {\bf{e}}_j \right ) \right \|_{\mathsf{A}_{_h}}^2 + \left \| \mathsf{I}_\mathsf{R} {\bf{v}}_{_h} + {\bf{v}}_{_H} +  \mathsf{P}_{0} \left ( {\bf{v}}_{_h} - \mathsf{I}_{\mathsf{P}} {\bf{v}}_{_H} \right) \right \|_{\mathsf{A}_{_H}}^2,
\end{eqnarray*}
where $(\widehat{u}_1,\ldots,\widehat{u}_{n_{_h}})^T = {\bf{v}}_{_h} - \mathsf{I}_{\mathsf{P}} {\bf{v}}_{_H}$. Next, we estimate the two parts in the above identity separately. 

It is straight-forward to  observe that
\begin{eqnarray*}
\|\mathsf{P}_0 ({\bf{v}}_{_h} - \mathsf{I}_{\mathsf{P}} {\bf{v}}_{_H}) \|_{\mathsf{A}_{_H}}^2 &=& (\mathsf{P}_0 ({\bf{v}}_{_h} - \mathsf{I}_{\mathsf{P}} {\bf{v}}_{_H}),\mathsf{P}_0 ({\bf{v}}_{_h} - \mathsf{I}_{\mathsf{P}} {\bf{v}}_{_H}))_{\mathsf{A}_{_H}} \\ 
&=& (\mathsf{I}_{\mathsf{R}} \mathsf{A}_{_h} ({\bf{v}}_{_h} - \mathsf{I}_{\mathsf{P}} {\bf{v}}_{_H}), \mathsf{A}_{_H}^{-1} \mathsf{I}_{\mathsf{R}} \mathsf{A}_{_h} ({\bf{v}}_{_h} - \mathsf{I}_{\mathsf{P}} {\bf{v}}_{_H})) \\
&=& (\mathsf{A}_{_h} ({\bf{v}}_{_h} - \mathsf{I}_{\mathsf{P}} {\bf{v}}_{_H}), \mathsf{I}_{\mathsf{P}} \mathsf{A}_{_H}^{-1} \mathsf{I}_{\mathsf{R}} \mathsf{A}_{_h} ({\bf{v}}_{_h} - \mathsf{I}_{\mathsf{P}} {\bf{v}}_{_H})) \\
&\leq& \|{\bf{v}}_{_h} - \mathsf{I}_{\mathsf{P}} {\bf{v}}_{_H} \|_{\mathsf{A}_{_h}}^2, \qquad \mbox{ due to the inequality } (\ref{trivin}). 
\end{eqnarray*}
Therefore, we obtain that   
\begin{eqnarray}
\left \| \mathsf{I}_\mathsf{R} {\bf{v}}_{_h} + {\bf{v}}_{_H} + \mathsf{P}_{0} \left ( {\bf{v}}_{_h} - \mathsf{I}_{\mathsf{P}}{\bf{v}}_{_H}\right )  \right \|_{\mathsf{A}_{_H}}^2 &\lesssim& 
\| \mathsf{I}_{\mathsf{P}} \mathsf{I}_\mathsf{R} {\bf{v}}_{_h} + \mathsf{I}_{\mathsf{P}} {\bf{v}}_{_H} \|_{\mathsf{A}_{_h}}^2 + \|\mathsf{P}_{0} \left ( {\bf{v}}_{_h} - \mathsf{I}_{\mathsf{P}}{\bf{v}}_{_H}\right ) \|_{\mathsf{A}_{_H}}^2 
\nonumber \\ 
&\lesssim& \left \| \mathsf{I}_{\mathsf{P}} \mathsf{I}_\mathsf{R} {\bf{v}}_{_h} + {\bf{v}}_{_h} \right \|_{\mathsf{A}_{_h}}^2 + \|  {\bf{v}}_{_h}  - \mathsf{I}_{\mathsf{P}} {\bf{v}}_{_H} \|_{\mathsf{A}_{_h}}^2
\nonumber \\ 
&=& \|I_H^h v_{_h}+ v_{_h}\|_1^2 + \|v_{_h} - Q_H v_{_h}\|_1^2. \label{eqn:part1}
\end{eqnarray}

Let $\Omega_i = {\rm supp}\, \phi_i$ and $\mathsf{A}_{_h} = (a_{ij})$ with $a_{ij} = a(\phi_j,\phi_i)$. Using the Cauchy-Schwarz inequality, we have the following estimate
\begin{eqnarray*}
\sum_{i=1}^{n_{_h}} \left \| \mathsf{P}_i \left ( \sum_{j=i}^{n_{_h}} \widehat{u}_j \right ) \right \|_{\mathsf{A}_{_h}}^2 &=&
\sum_{i=1}^{n_{_h}}\left ( \left ({\bf{e}}_i^T \mathsf{A}_{_h}
{\bf{e}}_i \right )^{-1} {\bf{e}}_i^T \mathsf{A}_{_h} \Big (\sum_{j=i}^{n_{_h}}\widehat{u}_j \Big ) {\bf{e}}_i^T
\mathsf{A}_{_h} \Big ( \sum_{j=i}^{n_{_h}} \widehat{u}_j \Big ) \right )
\\
&=& \sum_{i=1}^{n_{_h}} a(\phi_i,\phi_i)^{-1} a\left (\phi_i,\sum_{j=i}^{n_{_h}} \widehat{u}_j\phi_j \right )^2 \\
&\leq&
\sum_{i=1}^{n_{_h}} \int_{\Omega_i} \left | \nabla 
\sum_{j=i}^{n_{_h}} \widehat{u}_j \phi_j \right |^2 dx 
\lesssim \sum_{i=1}^{n_{_h}} \sum_{j \in N_k(i)} h^{d-2} \widehat{u}_j^2, 
\end{eqnarray*}
where $d$ is the dimension and $N_k(i) = \{j \in \{1,\ldots,n_{_h}\} : \Omega_j \cap \Omega_i \neq \emptyset \}$. Norm equivalence leads to the following estimate that 
\begin{equation*}
\sum_{j \in N_k(i)} h^{d} \widehat{u}_j^2 \lesssim \| v_{_h} - Q_H v_{_h} \|_{0,\Omega_i}^2. 
\end{equation*}
Therefore, we arrive at the conclusion that 
\begin{eqnarray}
\sum_{i=1}^{n_{_h}} \left \| \mathsf{P}_i \left ( \sum_{j=i}^{n_{_h}} \widehat{u}_j \right ) \right \|_{\mathsf{A}_{_h}}^2 
&\lesssim& h^{-2} \sum_{i=1}^{n_{_h}} \|v_{_h} - Q_H v_{_h}\|^2_{0,\Omega_i}\lesssim h^{-2} \|v_{_h} - Q_Hv_{_h}\|_0^2
\nonumber \\ 
& = & h^{-2} \|(v_{_h} + I_H^h v_{_h}) - Q_H ( v_{_h} + I_H^h v_{_h}) \|_0^2
\nonumber \\ 
&\lesssim& \|(I - Q_H)(v_{_h} + I_H^h v_{_h})\|_1^2. \label{eqn:part2}
\end{eqnarray}
By combining \eqref{eqn:part1} and \eqref{eqn:part2}, we obtain the desired estimate for $K$, which completes the proof.
\end{proof}


\section{Block preconditioners for the Stokes equation}\label{sec:mfem}

In this section, we consider efficient Poisson-based preconditioning techniques for the Stokes system with no-slip boundary condition: Find velocity ${\bf u}$ and pressure $p$, such that
\begin{equation}\label{eqn:stokes}
\left\{
\begin{array}{rcl}
- \Delta {\bf u} - \nabla p&=&{\bf f} \qquad \mbox{in }\Omega 
\\
\nabla\cdot {\bf u}&=& 0 \qquad \mbox{in }\Omega 
\\
{\bf u}&=&{\bf 0} \qquad \mbox{on }\partial\Omega,
\end{array}
\right.
\end{equation}
where $\Omega \subset \mathbb{R}^d \; (d=2,3)$  is a bounded polygonal domain and ${\bf f}$ is a given function. 

There have been extensive discussions on discretizations of the Stokes equation. In this paper, we shall focus on the mixed finite element discretizations; see, for example, \cite{Temam.R1977,Girault.V;Raviart.P1986,Brezzi.F;Fortin.M1991}. When solving this problem with mixed finite element methods (MFEM), a pair of discrete function spaces for velocity and pressure must be chosen carefully so that it is stable, i.e., satisfying so-called the inf-sup condition. There are a number of important classes of stable pairs. In particular, a family of Hood--Taylor finite elements~\cite{Taylor1973} that approximates the velocity by continuous piecewise $k$-th order polynomials ($P^{k,0}$) and the pressure by continuous piecewise $(k-1)$-th order polynomials ($P^{k-1,0}$) with $k \geq 2.$ is known to be stable for full three dimensional Stokes equation \cite{Boffi1997}. Another important stable elements shown only in 2D, is the Scott--Vogelius elements~\cite{Scott.L;Vogelius.M1985} with $k \geq 4$, which approximates the velocity by continuous piecewise $k$-th order polynomials ($P^{k,0}$) and the pressure by discontinuous piecewise $(k-1)$-th order polynomials ($P^{k-1,-1}$). These pairs of mixed finite elements are very promising because they preserve the incompressibility condition, namely, discrete divergence free condition in the strong sense. 

Assume that the coefficient matrix arising in mixed finite element discretizations of \eqref{eqn:stokes} can be written as
\[
F :=
 \left ( \begin{array}{cc} A & B^T \\ B & 0 \end{array}
 \right)  
\]
Here, $B$ is the discrete divergence operator, i.e., $-\nabla_h\cdot$; and, $A$ is the block diagonal matrix with the discrete Laplace matrix $-\Delta_h$ on its diagonal. Let $x_\vu$ and $x_p$ be the unknown vectors of the velocity field and pressure, respectively. Then we need to solve the system of linear equations
\begin{equation}\label{saddlepoint}
F x = b \qquad \text{or} \qquad F \, \Big[ \begin{array}{c} x_\vu \\ x_p \end{array} \Big] =  \Big[\, b \,\Big].
\end{equation}

$F$ is symmetric and positive semidefinite and we can apply Krylov subspace methods for indefinite problems such as the minimal residual (MINRES) method~\cite{Paige.Saunders1975} and the generalized minimal residual (GMRES) method~\cite{Saad.Y;Schultz.M1986} to solve \eqref{saddlepoint}. It is well-known that the convergence rate of such methods is governed by three parameters: the condition numbers of $A$, $B$ and the relative scaling between them. Generally, $\text{cond}(A) = O(h^{-2})$, therefore, a direct application of these methods will yield a slow convergence and the main difficulty in solving the Stokes equations lies in constructing ``good'' preconditioners for the elliptic operator $A$, which is typically given by the discrete elliptic operator with higher order finite elements. 

A lot of efforts have been devoted on solving the saddle point problems arising from mixed finite element methods for the Stokes equation; see \cite{elman2005finite,Benzi2005,Larin.Reusken2008} and references therein for details. A few efficient multigrid-type solution methods have been proposed for the Stokes equation~\cite{Brandt1979,Wittum1989,Vanka1986,Braess1997a}. In this paper, we focus on the Poisson-based block diagonal and block triangular preconditioners~\cite{Bramble.Pasciak1988,Rusten.T;Winther.R1992}. Recently, parallel version of these preconditioning algorithms attracted a lot of interests due to their efficiency and easiness for implementation; see for example~\cite{May2008,Geenen2009}.

Denote the Schur complement by $S:=B A^{-1} B^T$.
The block triangular factorizations of $F$,
\begin{eqnarray}
 \left ( 
 \begin{array}{cc} A & B^T \\ B & 0 \end{array}
 \right) 
&=&
 \left(
\begin{array}{cc}
I_\vu & 0 \\
B A^{-1} & I_p 
\end{array}
\right)
\left(
\begin{array}{cc}
A & B^T \\
0 & -S 
\end{array}
\right)
\nonumber \\
&=&
\left(
\begin{array}{cc}
I_\vu & 0 \\
B A^{-1} & I_p 
\end{array}
\right)
\left(
\begin{array}{cc}
A & 0 \\
0 & -S 
\end{array}
\right)
\left(
\begin{array}{cc}
I_\vu & A^{-1}B^T \\
0 & I_p 
\end{array}
\right), 
\qquad
\end{eqnarray}
motivate a block upper triangular preconditioner~\cite{Bramble.Pasciak1988}
\begin{equation}\label{precond_t}
Q_t =
\left (
\begin{array}{cc} A & B^T \\ 0 & -S \end{array}
\right)^{-1}
\end{equation}
and an even simpler block diagonal preconditioner~\cite{Rusten.T;Winther.R1992}
\begin{equation}\label{precond_d}
Q_d =
\left (
\begin{array}{cc} A & 0 \\ 0 & -S \end{array}
\right)^{-1}.
\end{equation}

In either $Q_t$ or $Q_d$, it requires to obtain certain approximations to $A^{-1}$ and $S^{-1}$. Since $A$ is a discretization of the Laplace operator, we form an approximation to $A^{-1}$ by applying one multilevel V-cycle to $A$. As for the Schur complement part, we approximate $S$ with the pressure mass matrix $M_p$ and solve $M_p$ equation using the conjugate gradient method with diagonal preconditioner. We note that there are different ways to form the preconditioning actions based on \eqref{precond_t} and \eqref{precond_d}; see~\cite{Benzi2005}. 

\section{Numerical experiments}\label{sec:numer}

\newlength{\arrayrulewidthOriginal}
\newcommand{\Cline}[2]{%
  \noalign{\global\setlength{\arrayrulewidthOriginal}{\arrayrulewidth}}%
  \noalign{\global\setlength{\arrayrulewidth}{#1}}\cline{#2}%
  \noalign{\global\setlength{\arrayrulewidth}{\arrayrulewidthOriginal}}}
  
In this section, we test the performance of the proposed Geometric-Algebraic Multigrid (GAMG) method (Algorithm~\ref{gamg} with presmoothing; see Remark~\ref{rem:2-1vsGAMG}) and compare it with the corresponding AMG method with focus on their robustness, efficiency, and parallel weak scalability. For this purpose, we use two simple test problems---one is the 3D Poisson equation and the other is the 3D Stokes equation---on a unit cube with unstructured tetrahedral meshes. We pay special attention to the performance of both methods for higher-order finite element discretizations. 

\subsection{Implementation}

All numerical tests are carried out on the LSSC-I\!I\!I cluster at the State Key Laboratory of Scientific and Engineering Computing (LSEC), Chinese Academy of Sciences. The LSSC-I\!I\!I cluster has 282 computing nodes: Each node has two Intel Quad Core Xeon X5550 2.66GHz processors and 24GB shared memory; all nodes are connected via Gigabit Ethernet and DDR InfiniBand. To make a fair comparison, we use zero right-hand side and start from a random initial guess in our tests. In this section, ``\#It'' denotes the number of iterations, ``DOF'' denotes the degree of freedom, ``CPU'' denotes the computation wall time in seconds, and ``RAM'' denotes the memory usage per processor in MB.  

Our implementation is based on several open-source numerical packages. The finite element discretization for the Poisson equation and the Stokes equation is implemented using PHG~\cite{phg,Zhang2005}. PHG is a toolbox for developing parallel adaptive finite element programs on unstructured tetrahedral meshes and it is under active development at LSEC. PHG is also employed to build the two-level GAMG setting, namely to generate the transfer operators (prolongation and restriction). 

The solvers are implemented using PETSc~\cite{petsc} and BoomerAMG in {\it hypre}~\cite{hypre}. Due to limited space, we only report the results for the Flexible GMRES (FGMRES) method~\cite{ysaad1993,Saad2003} in PETSc and BoomerAMG with the PMIS method~\cite{Sterck2006} for coarsening, the Extended+i+cc method~\cite{Sterck.H;Falgout.R;Nolting.J;Yang.U2008} for interpolation, and the hybrid Gauss-Seidel method for smoothing. This particular setting provides good efficiency and scalability for the linear systems in our numerical tests. For more numerical tests for various choices of iterative methods and different types of AMG methods, we refer to Lee, Leng and Zhang~\cite{Lee2012}. 

\begin{remark}[Number of Smoothing Sweeps]\rm
In multigrid method, there is a trade-off in number of smoothing sweeps used in each cycle. More smoothing sweeps in one cycle will cost more computation time in each multigrid cycle but may reduce total number of cycles. Another consideration about the number of smoothing sweeps is that for linear system that is harder to solve, more smoothing sweeps might lead to better convergence rate. In this paper, we use only one pre and post smoothing sweep in our experiments.
\end{remark}

   
\subsection{Test Problem 1---the Poisson equation}\label{test:Poisson}

The GAMG solver for the Poisson equation is implemented as follows: We first pass the linear systems to the FGMRES iterative method of PETSc, then we use one multilevel V-cycle as the preconditioner for FGMRES. In each multilevel cycle, the coarse level problem (corresponding to the $P^{1,0}$ finite element space) is solved with BoomerAMG in {\it hypre}. The AMG preconditioner for the Poisson problem is simple, we pass the linear system to the FGMRES method of PETSc and employ BoomerAMG  as a preconditioner. In both methods, the stopping criteria is that the relative residual is less than $10^{-6}$. 

In Table~\ref{tab:Poisson-hypre} in \S\ref{sec:intro}, we have shown that the AMG method could be very sensitive to the strength threshold $\theta$ when applied to the linear systems arising in higher order finite element discretizations. The strength threshold or strong threshold determines strength of connections, i.e., a point (variable) $i$ strongly depends on $j$ if 
\[
- a_{i,j} > \theta \max_{k \neq i} (- a_{i,k}).
\]
The default value of $\theta$ in BoomerAMG is $0.25$, which usually works well for 2D Laplace operators and a larger value, like $0.5$, is suggested for 3D Laplace operators. But neither of them works for the $P^{4,0}$ finite element discretization. Hence we start by testing both AMG and GAMG methods for various values of $\theta$ and report CPU time and memory consumption in Table~\ref{tab:P2}, \ref{tab:P3}, \ref{tab:P4}. 

We notice that: 
(1) The AMG method under consideration performs reasonably well even for high-order finite element methods. 
(2) However, it is very sensitive to the strength threshold $\theta$ for the fourth order finite element method.
(3) The GAMG method, on the other hand, is very robust with respect to $\theta$. And, in general, it converges faster (from $1.5$ times up to $30$ times, see Table~\ref{tab:Poisson-speedup}) and consumes less memory (by $10\%$ to $50\%$) compared with the AMG method. 
(4) For large 3D linear systems arising from higher-order finite element discretizations, large strength threshold often works much better. In the rest of the comparisons, we will fix the parameter $\theta = 0.8$, which is the best choice for AMG, but not necessarily for GAMG.


\begin{table}[h!!]
\def\temptablewidth{1\textwidth}
\renewcommand{\arraystretch}{1.12}
\centering
\caption{\rm Iteration number, CPU time, and memory usage of the AMG and GAMG preconditioned Krylov subspace method for the 3D Poisson equation on unstructured tetrahedral mesh ($P^{2,0}$ finite element, about 430K DOF per processing core). } \label{tab:P2}
\begin{tabular*}{\temptablewidth}{|r|rrr|rrr|rrr|}
    \cline{1-10}
\multirow{2}{*}{ Method ($\theta$)\;\;} & \multicolumn{3}{c|}{1 Core} &  \multicolumn{3}{c|}{8 Cores} &  \multicolumn{3}{c|}{64 Cores} 
\\ \cline{2-10}
 	             & \#It & CPU & RAM & \#It & CPU & RAM & \#It & CPU & RAM
\\
\rowcolor[gray]{.85} 
AMG\,(0.2)         & 7 & 11.75 & 1141 & 6 & 22.82 & 1424 & 6 & 32.71 & 1320
\\ 
GAMG\,(0.2)	  & 7 & 4.16 & 790 & 7 & 9.25 & 1170 & 7 & 8.57 & 1072
\\ 
\rowcolor[gray]{.85} 
AMG\,(0.4)         & 7 & 8.00 & 1124 & 7 & 18.13 & 1428 & 7 & 25.50 & 1294
\\ 
GAMG\,(0.4)	  & 7 & 4.16 & 790 & 7 & 9.39 & 1170 & 7 & 8.79 & 1073
\\ 
\rowcolor[gray]{.85} 
AMG\,(0.6)         & 8 & 5.92 & 1080 & 8 & 15.06 & 1381 & 8 & 20.29 & 1255
\\
GAMG\,(0.6)	  & 7 & 4.09 & 790 & 7 & 8.55 & 1174 & 7 & 8.19 & 1074
\\ 
\rowcolor[gray]{.85} 
AMG\,(0.8)         & 11 & 5.10 & 1022 & 13 & 14.64 & 1280 & 13 & 17.85 & 1187
\\
GAMG\,(0.8)	  & 9 & 4.35 & 790 & 9 & 9.13 & 1168 & 9 & 9.49 & 1072
\\ \cline{1-10}  
\end{tabular*}
\end{table}


\begin{table}[h!!]
\def\temptablewidth{1\textwidth}
\renewcommand{\arraystretch}{1.12}
\centering
\caption{\rm Iteration number, CPU time, and memory usage of the AMG and GAMG preconditioned Krylov subspace method for the 3D Poisson equation on unstructured tetrahedral mesh ($P^{3,0}$ finite element, about 650K DOF per processing core).} \label{tab:P3}
\begin{tabular*}{\temptablewidth}{|r|rrr|rrr|rrr|}
    \cline{1-10}
\multirow{2}{*}{ Method ($\theta$)\;\;} & \multicolumn{3}{c|}{1 Core} &  \multicolumn{3}{c|}{8 Cores} &  \multicolumn{3}{c|}{64 Cores} 
\\ \cline{2-10}
 	             & \#It & CPU & RAM & \#It & CPU & RAM & \#It & CPU & RAM
\\
\rowcolor[gray]{.85} 
AMG\,(0.2)         & 8 & 32.24 & 2345 & 8 & 42.98 & 2133 & 7 & 75.84 & 2419
\\   
GAMG\,(0.2)	  & 12 & 13.14 & 1216 & 12 & 14.11 & 1410 & 11 & 19.46 & 1522
\\ 
\rowcolor[gray]{.85} 
AMG\,(0.4)         & 9 & 23.94 & 2278 & 8 & 35.15 & 2219 & 8 & 65.24 & 2507
\\   
GAMG\,(0.4)	  & 12 & 13.16 & 1217 & 12 & 14.94 & 1412 & 11 & 19.99 & 1523
\\ 
\rowcolor[gray]{.85} 
AMG\,(0.6)         & 10 & 16.76 & 1929 & 10 & 27.02 & 2067 & 10 & 48.07 & 2291
\\   
GAMG\,(0.6)	  & 12 & 13.12 & 1222 & 12 & 14.86 & 1412 & 11 & 20.21 & 1523
\\ 
\rowcolor[gray]{.85} 
AMG\,(0.8)         & 13 & 13.74 & 1696 & 14 & 23.29 & 1831 & 15 & 38.54 & 2137
\\   
GAMG\,(0.8)	  & 12 & 13.08 & 1216 & 12 & 14.86 & 1411 & 11 & 20.19 & 1525
\\ \cline{1-10}  
\end{tabular*}
\end{table}


\begin{table}[h!!]
\def\temptablewidth{1\textwidth}
\renewcommand{\arraystretch}{1.12}
\centering
\caption{\rm Iteration number, CPU time, and memory usage of the AMG and GAMG preconditioned Krylov subspace method for the 3D Poisson equation on unstructured tetrahedral mesh ($P^{4,0}$ finite element, about 460K DOF per processing core).} \label{tab:P4}
\begin{tabular*}{\temptablewidth}{|r|rrr|rrr|rrr|}
    \cline{1-10}
\multirow{2}{*}{ Method ($\theta$)} & \multicolumn{3}{c|}{1 Core} &  \multicolumn{3}{c|}{8 Cores} &  \multicolumn{3}{c|}{64 Cores} 
\\ \cline{2-10}
 	             & \#It & CPU & RAM & \#It & CPU & RAM & \#It & CPU & RAM
\\
\rowcolor[gray]{.85} 
AMG\,(0.2)         & 19 & 42.47 & 1986 & 28 & 73.22 & 2043 & 93 & 438.8 & 2352
\\  
GAMG\,(0.2)	  & 16 & 16.45 & 1023 & 18 & 20.03 & 1117 & 17 & 25.77 & 1534
\\ 
\rowcolor[gray]{.85} 
AMG\,(0.4)         & 16 & 22.12 & 1560 & 36 & 57.40 & 1840 & $>$500 & $>$800 & 2258
\\
GAMG\,(0.4)	  & 16 & 16.41 & 1023 & 18 & 20.05  & 1117 & 17 & 25.10 & 1535
\\ 
\rowcolor[gray]{.85} 
AMG\,(0.6)         & 17 & 16.44 & 1679 & 25 & 33.83 & 1683 & 245 & 356.2 & 2085
\\
GAMG\,(0.6)	  & 16 & 16.42 & 1023 & 18 & 19.07 & 1117 & 17 & 25.64 & 1539
\\ 
\rowcolor[gray]{.85} 
AMG\,(0.8)         & 17 & 13.08 & 1319 & 19 & 22.79 & 1555 & 19 & 36.50 & 1971
\\
GAMG\,(0.8)	  & 16 & 16.47 & 1023 & 18 & 20.93 & 1118 & 17 & 25.23 & 1534
\\ \cline{1-10}  
\end{tabular*}
\end{table}

\begin{table}[h!!]
\renewcommand{\arraystretch}{1.12}
\centering
\caption{\rm Speedup of GAMG compared with AMG for solving the discrete Poisson equation with finite element methods on unstructured tetrahedral meshes (solved using 64 processing cores).} \label{tab:Poisson-speedup}
\begin{tabular}{|cc|c|c|c|c|}
\cline{1-6}
Finite Element & DOF & $\theta = 0.2$ & $\theta = 0.4$ & $\theta = 0.6$ & $\theta = 0.8$ 
\\ 
\cline{1-6}
$P^{2,0}$ & 27M & 3.8 & 2.9 & 2.5 & 1.9
\\
\cline{1-6}
$P^{3,0}$ & 43M & 3.9 & 3.3 & 2.4 & 1.9
\\
\cline{1-6}
$P^{4,0}$ & 31M & 17.0 & 32.0 & 13.9 & 1.5
\\
\cline{1-6}
\end{tabular}
\end{table}

Now, let $L$ to be the number of levels in multilevel hierarchy and level $1$ is the finest level. The operator complexity, $C_{\mathop{op}} := \sum_{l=1}^L \mathop{nnz}(A_l) / \mathop{nnz}(A_1)$, is the ratio between the total number of nonzeros (nnz) of all levels and the number of nonzeros of the finest level. The operator complexity is an important indicator of expense of multigrid type methods, not only for the storage requirements of multilevel preconditioners, but also for computational complexity for applying them. In our experiments, we use the PMIS coarsening strategy and the Extended+i+cc interpolation method in both AMG and GAMG. As summarized in Table~\ref{tab:Poisson-Cop}, we notice that GAMG action is much cheaper than AMG. In fact, for $P^{3,0}$ and $P^{4,0}$, the operator complexity of GAMG is close to $1.0$. This is due to the coarse level space, $P^{1,0}$-finite element space, contains considerably less degree of freedom and gives much less number of nonzeros in the coefficient matrices. In Table~\ref{tab:Poisson-Cop}, the $P^{1,0}$--DOF column gives the degree of freedom for the corresponding coarse level. 

\begin{table}[h!!]
\renewcommand{\arraystretch}{1.15}
\centering
\caption{\rm Operator complexities of AMG and GAMG (PMIS and Extended+i+cc) in single core tests.} \label{tab:Poisson-Cop}
\begin{tabular}{|c|c|c|c|c|c|}
\cline{1-6}
\multirow{2}{*}{ Finite Element } & \multirow{2}{*}{ DOF } &  \multirow{2}{*}{ $P^{1,0}$--DOF } & \multicolumn{3}{c|}{$C_{\mathop{op}}$ } 
\\ 
\cline{4-6}
 &  &  & AMG & $P^{1,0}$--AMG & GAMG
\\ 
\cline{1-6}
$P^{1,0}$ & 429877 & --- & 2.00 & --- & ---
\\
\cline{1-6}
$P^{2,0}$ & 435825 & 59495 & 1.67 & 1.76 & 1.12
\\
\cline{1-6}
$P^{3,0}$ & 736608 & 31087 & 1.56 & 1.96 & 1.02
\\
\cline{1-6}
$P^{4,0}$ & 460899 & 8165 & 1.36 & 1.70 & 1.01
\\
\cline{1-6}
\end{tabular}
\end{table}

In the rest of this subsection, we consider weak scalability of the proposed GAMG method. The maximal number of processing cores is $1024$ (on $128$ nodes) and the maximal degree of freedom in our tests are about $5 \times 10^8$. We notice that both the AMG and GAMG preconditioned Krylov subspace methods yield good optimality and scalability; see Figures~\ref{fig:P2}, \ref{fig:P3}, and \ref{fig:P4}. We note that this comparison was done with the ``good'' parameter ($\theta=0.8$); otherwise, the performance of AMG will deteriorate quickly for $P^{4,0}$ finite element. 

\begin{figure}[h!!]
\centering
 \subfigure{
   \includegraphics[width=0.475\linewidth] {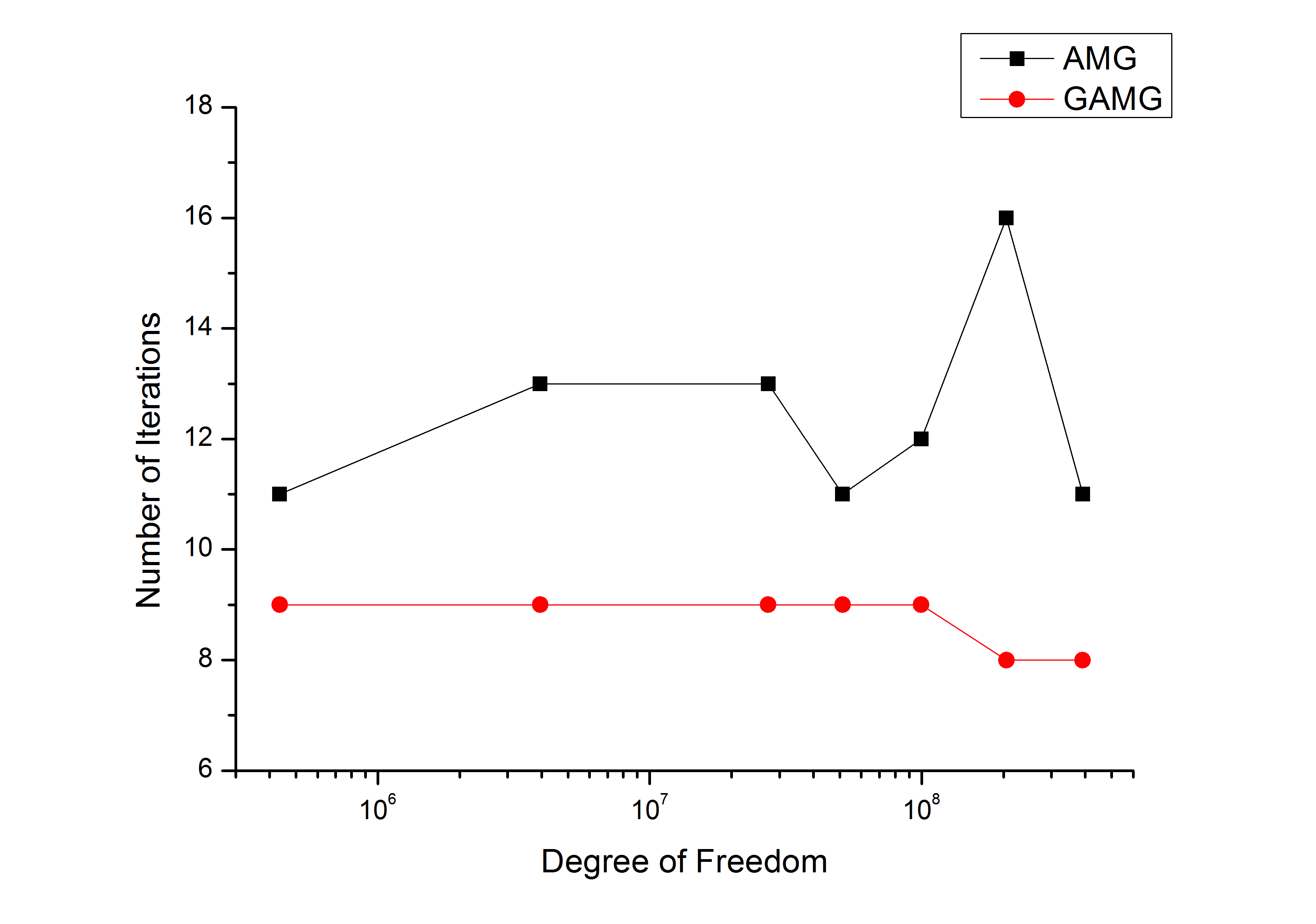}
 }
\subfigure{
   \includegraphics[width=0.475\linewidth] {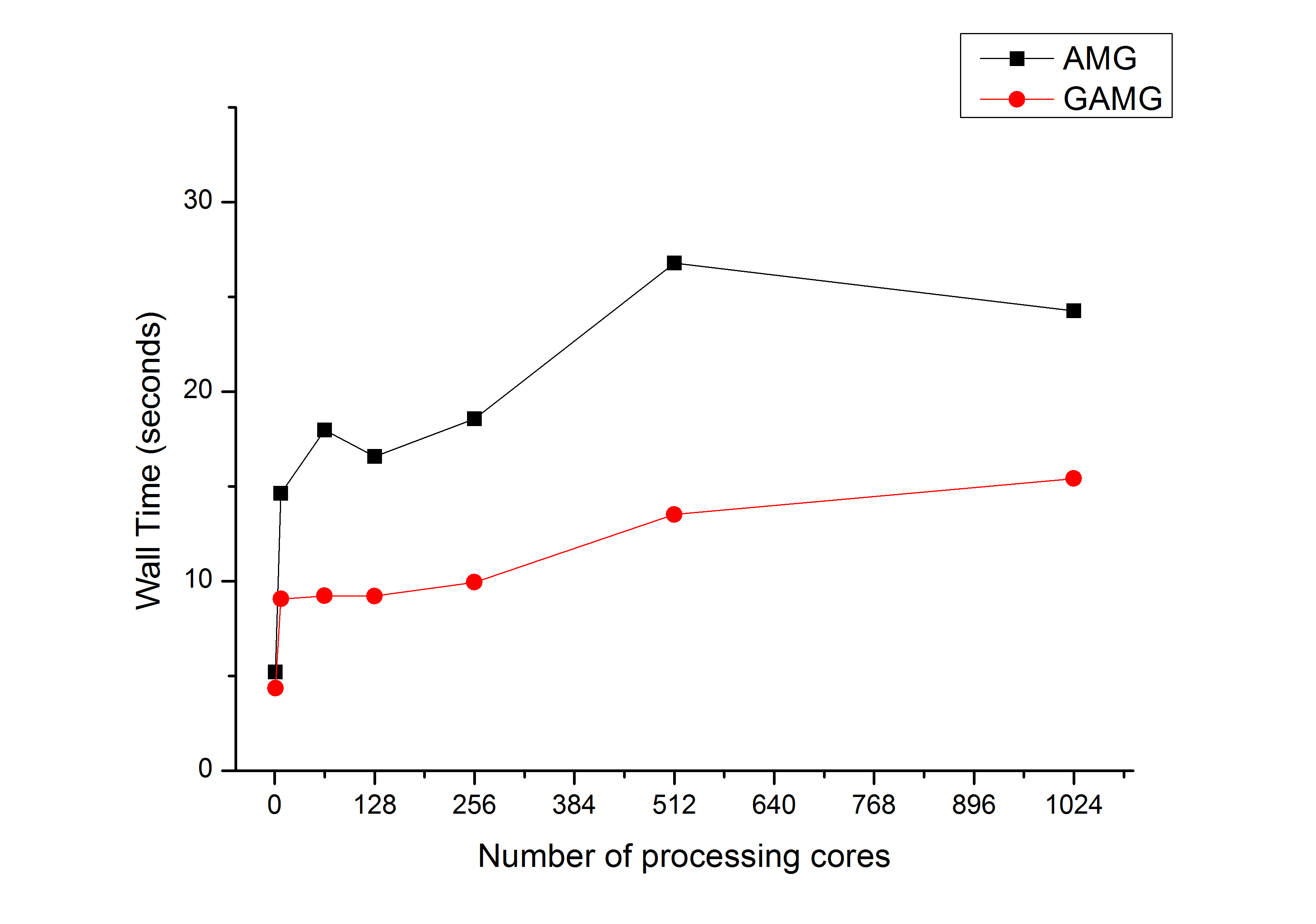}
 }
\caption{Parallel (weak) scalability of AMG and GAMG for $P^{2,0}$ FEM for the Poisson equation.}\label{fig:P2}
\end{figure}

\begin{figure}[h!!] 
   \centering
 \subfigure{
   \includegraphics[width=0.475\linewidth] {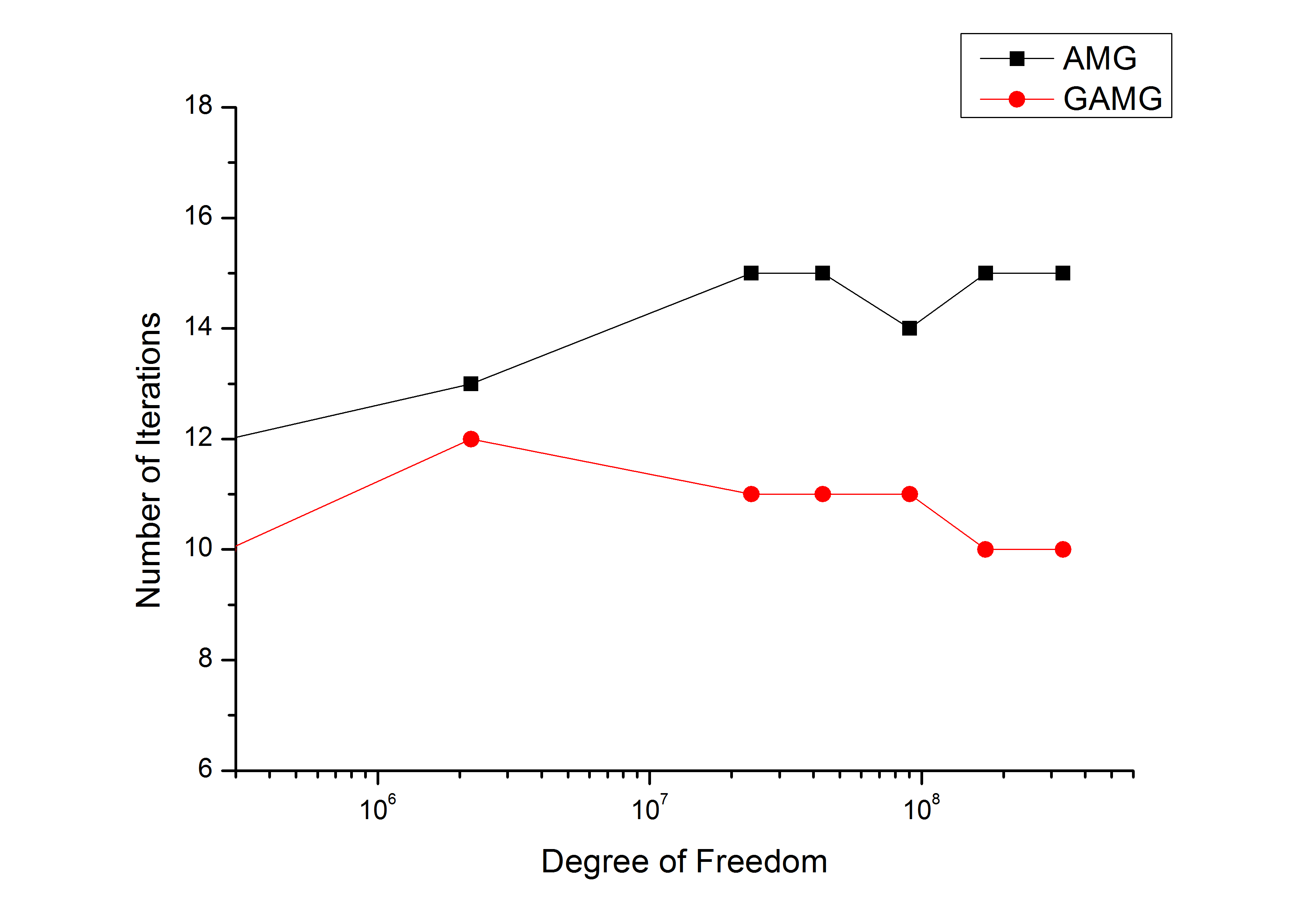}
 }
\subfigure{
   \includegraphics[width=0.475\linewidth] {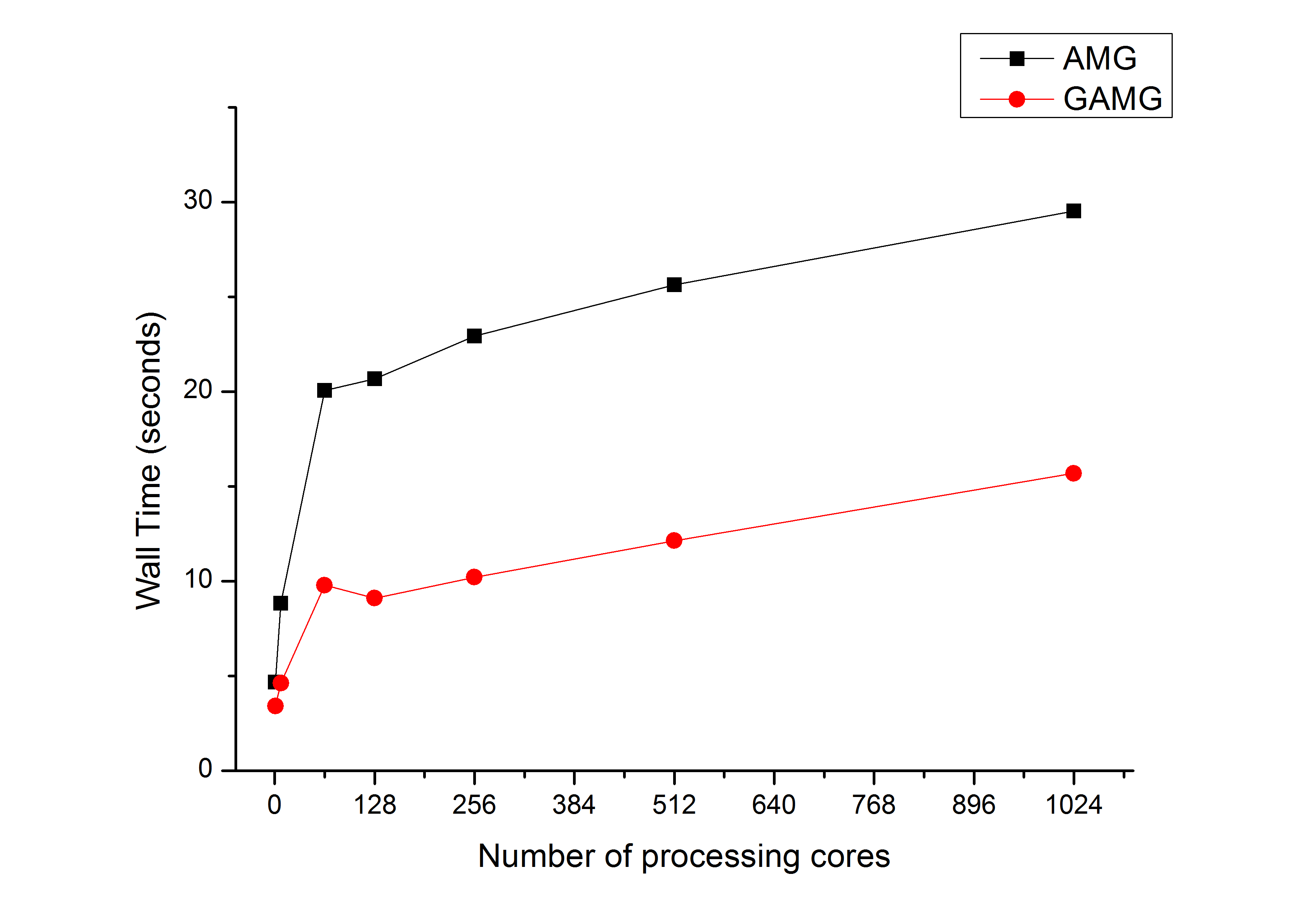}
 }
\caption{Parallel (weak) scalability of AMG and GAMG for $P^{3,0}$ FEM for the Poisson equation.}\label{fig:P3}
\end{figure}

\begin{figure}[h!!] 
   \centering
 \subfigure{
   \includegraphics[width=0.475\linewidth] {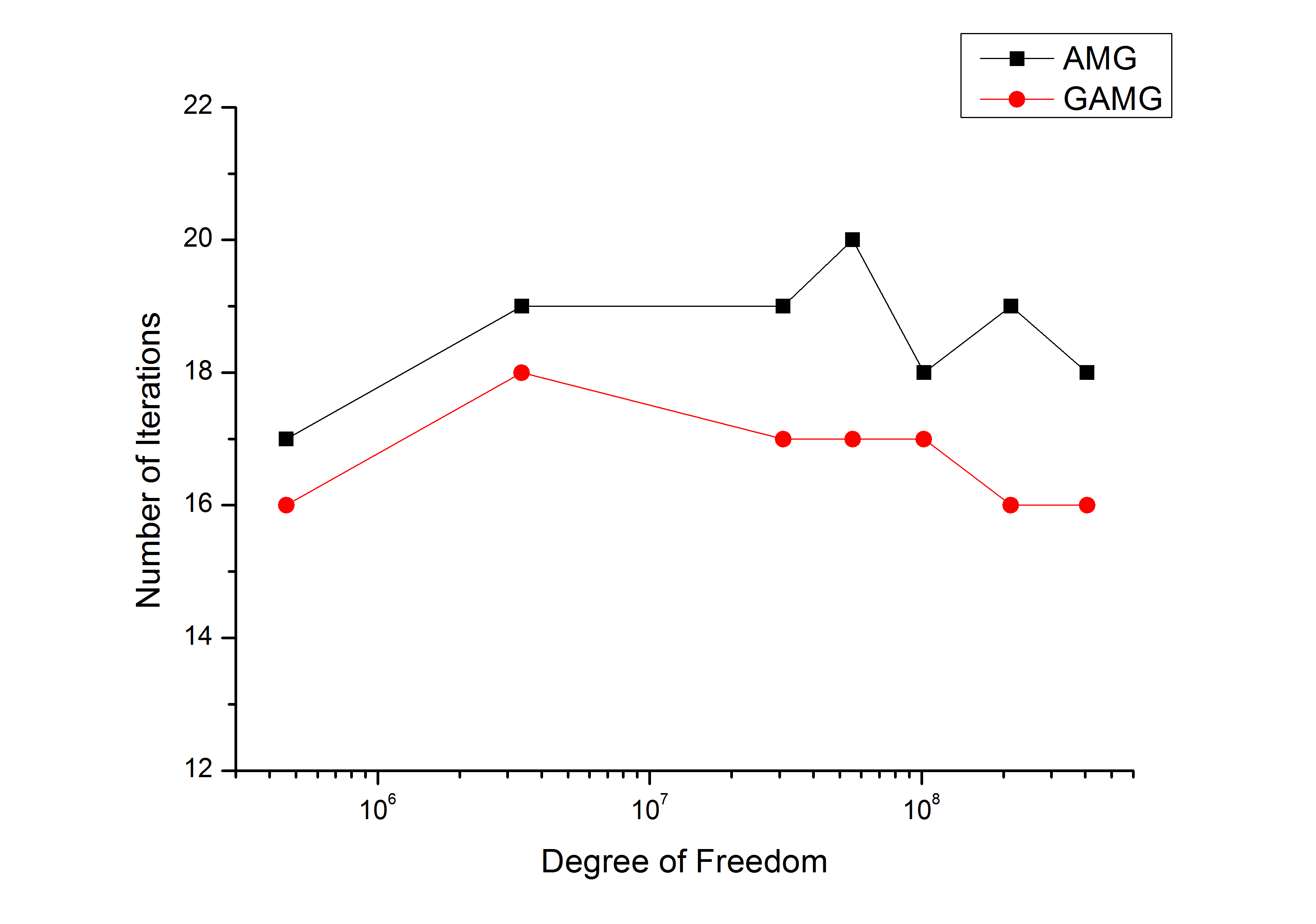}
 }
\subfigure{
   \includegraphics[width=0.475\linewidth] {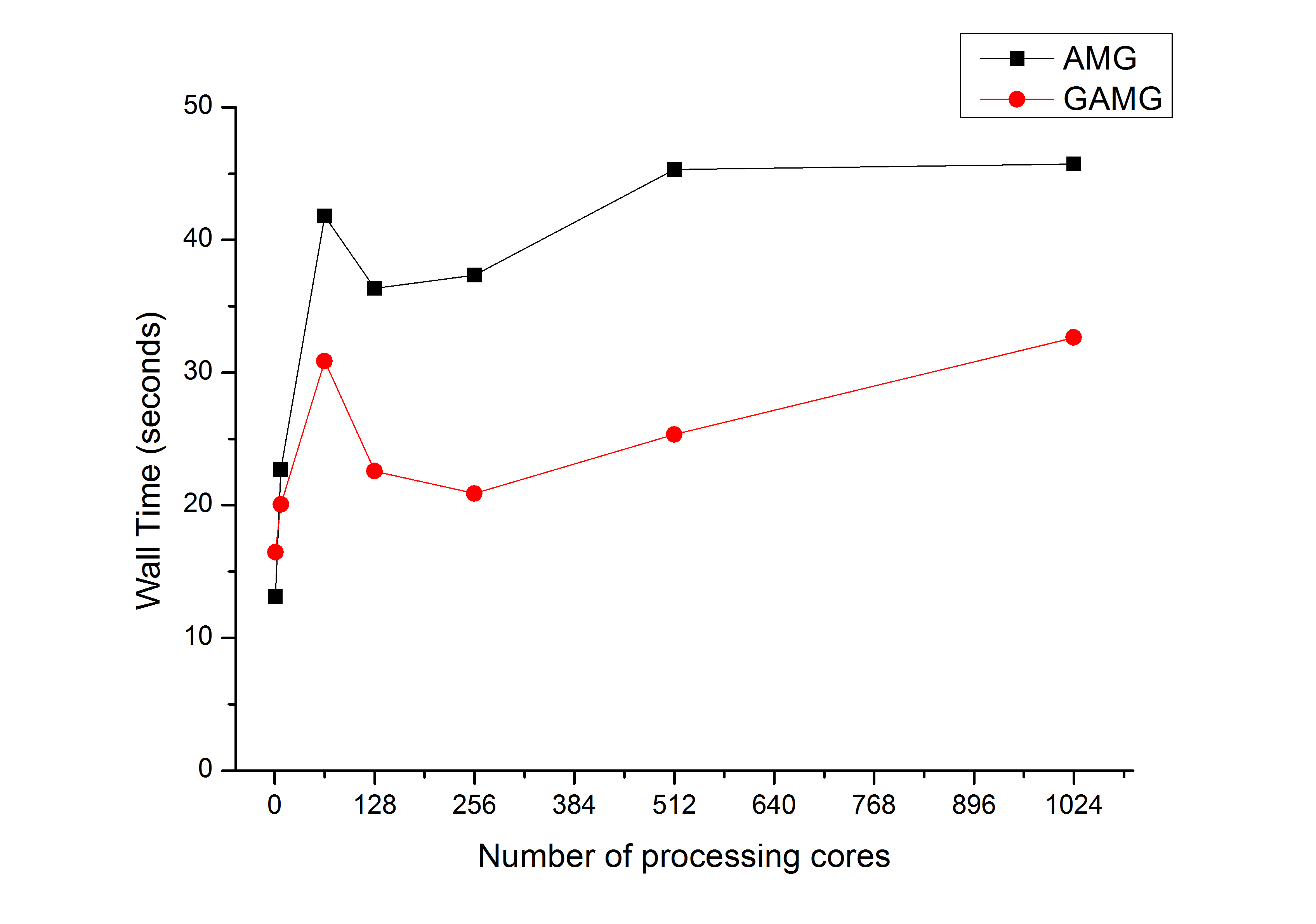}
 }
\caption{Parallel (weak) scalability of AMG and GAMG for $P^{4,0}$ FEM for the Poisson equation.}\label{fig:P4}
\end{figure}

\subsection{Test Problem 2---the Stokes equation}\label{test:Stokes}

In this section, we consider the Hood-Taylor family mixed finite element methods for the steady Stokes flow on a 3D lid driven cavity domain. Again, we test the AMG and GAMG methods with ``good'' strength threshold $\theta=0.8$. We choose to stop the outer FGMRES iteration if the relative residual is smaller than $10^{-8}$. 

Similar to the Poisson solver described in \S\ref{test:Poisson}, the GAMG solver for Stokes test is implemented as follows: We pass the linear systems to the FGMRES solver in PETSc and we apply the block triangular preconditioner $Q_t$ described in \S\ref{sec:mfem} for the FGMRES method. The performance of block diagonal preconditioner $Q_d$ can be found in~\cite{Lee2012}. For the lower-right block (corresponding to the Schur complement), we solve it with the diagonal preconditioned PCG method in PETSc. For the upper-left blocks (corresponding to the Poisson equation), we approximate it with one GAMG V-cycle for the discrete Poisson equation. The AMG solver for Stokes test is similar except that the upper-left block was solved with one AMG V-cycle.

From Figure~\ref{fig:P2-P1}, \ref{fig:P3-P2}, and \ref{fig:P4-P3}, we immediately notice that:
(1) In general, the AMG preconditioner performs reasonably well, even for high-order elements. 
(2) However, the convergence rate of AMG deteriorates with the size of problems and with the order of the mixed finite element. 
(3) The GAMG preconditioner yields much better convergence rate as well as scalability, especially for higher order elements. In particular, the iteration number does not increase as DOF increases. 

\begin{figure}[h!!] 
   \centering
 \subfigure{
   \includegraphics[width=0.475\linewidth] {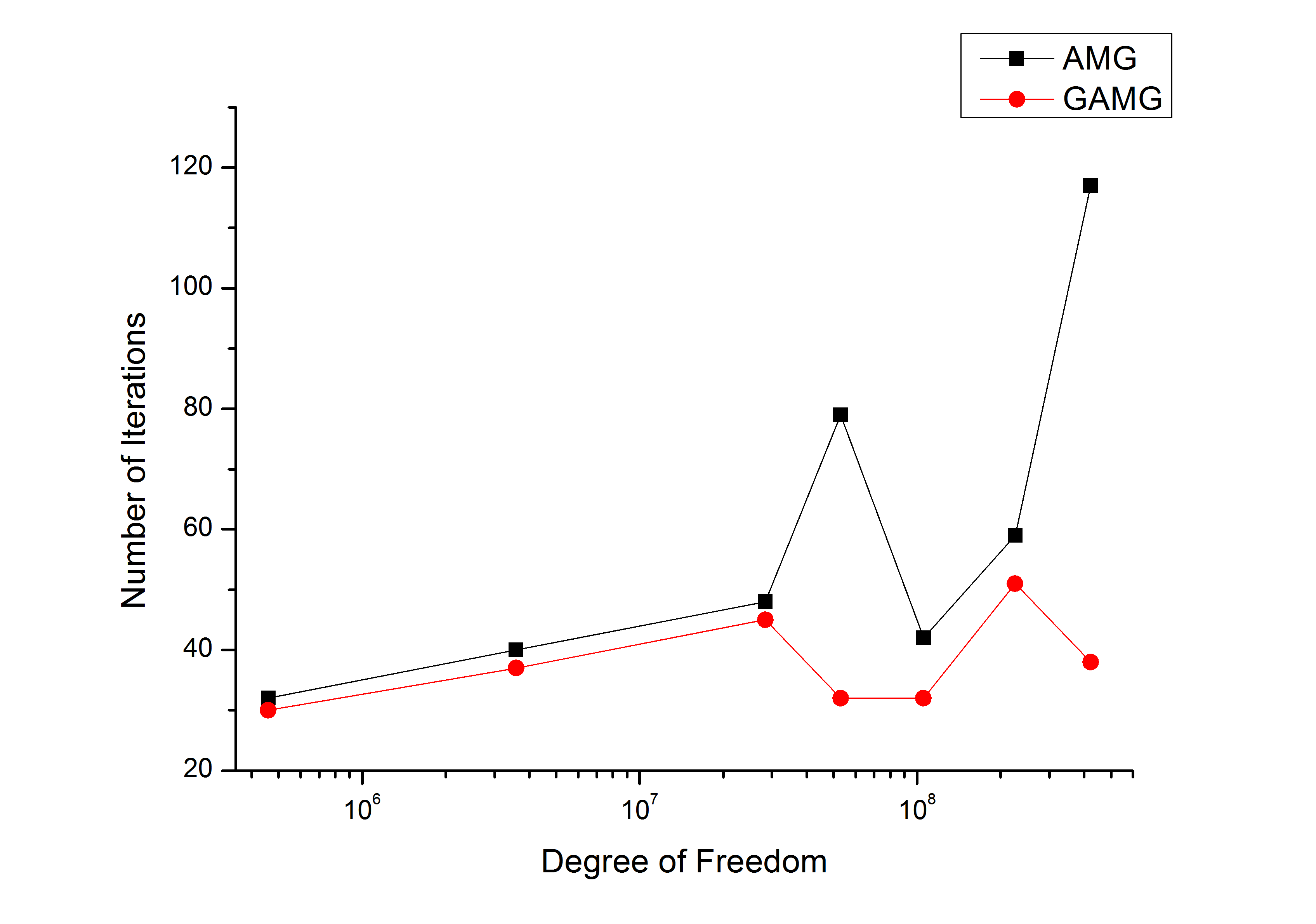}
 }
\subfigure{
   \includegraphics[width=0.475\linewidth] {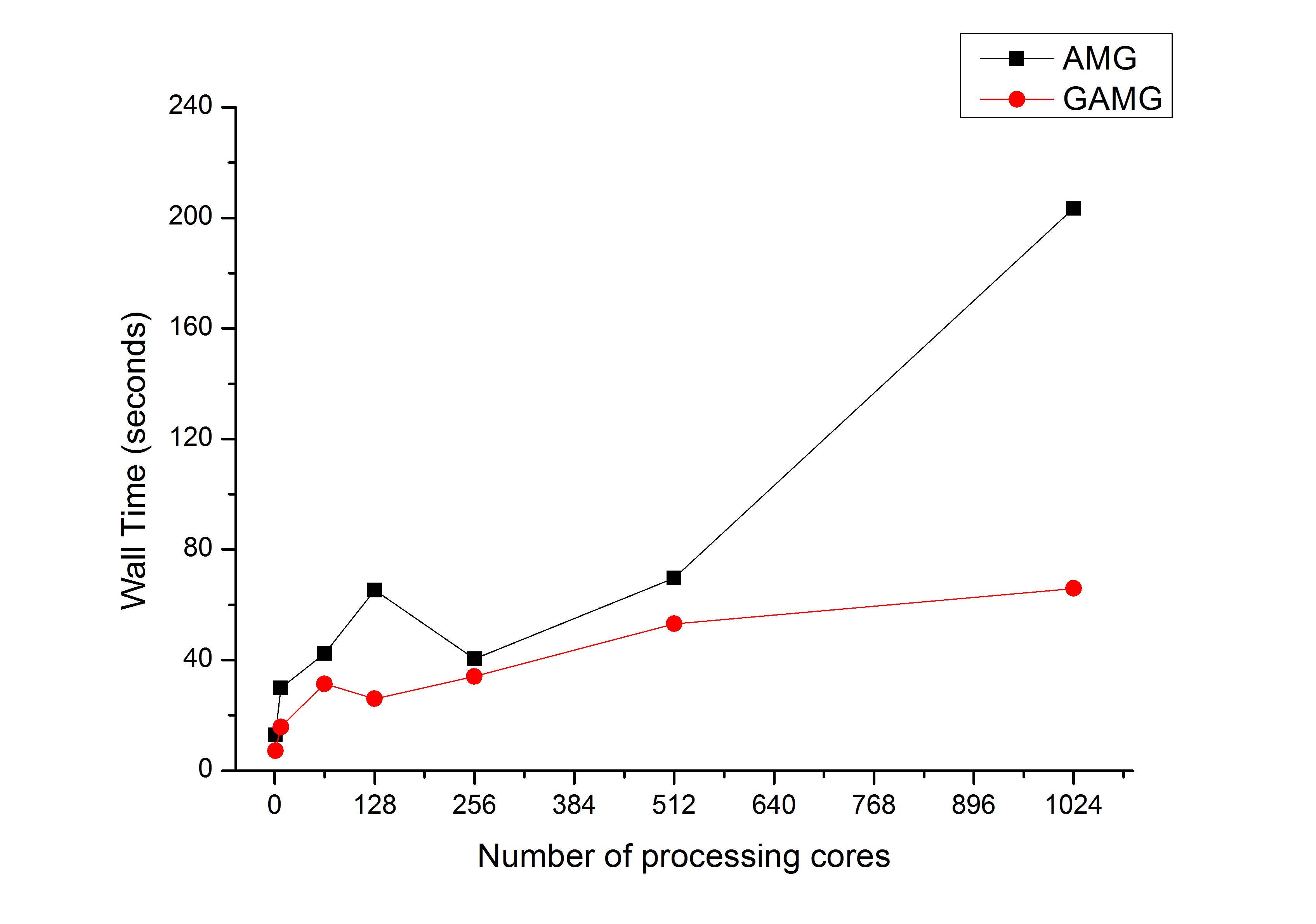}
 }
   \caption{Algorithm optimality and parallel (weak) scalability of AMG and GAMG preconditioned FGMRES methods for the $P^{2,0}-P^{1,0}$ Hood-Taylor mixed finite element.}
   \label{fig:P2-P1}
\end{figure}

\begin{figure}[h!!] 
   \centering
 \subfigure{
   \includegraphics[width=0.475\linewidth] {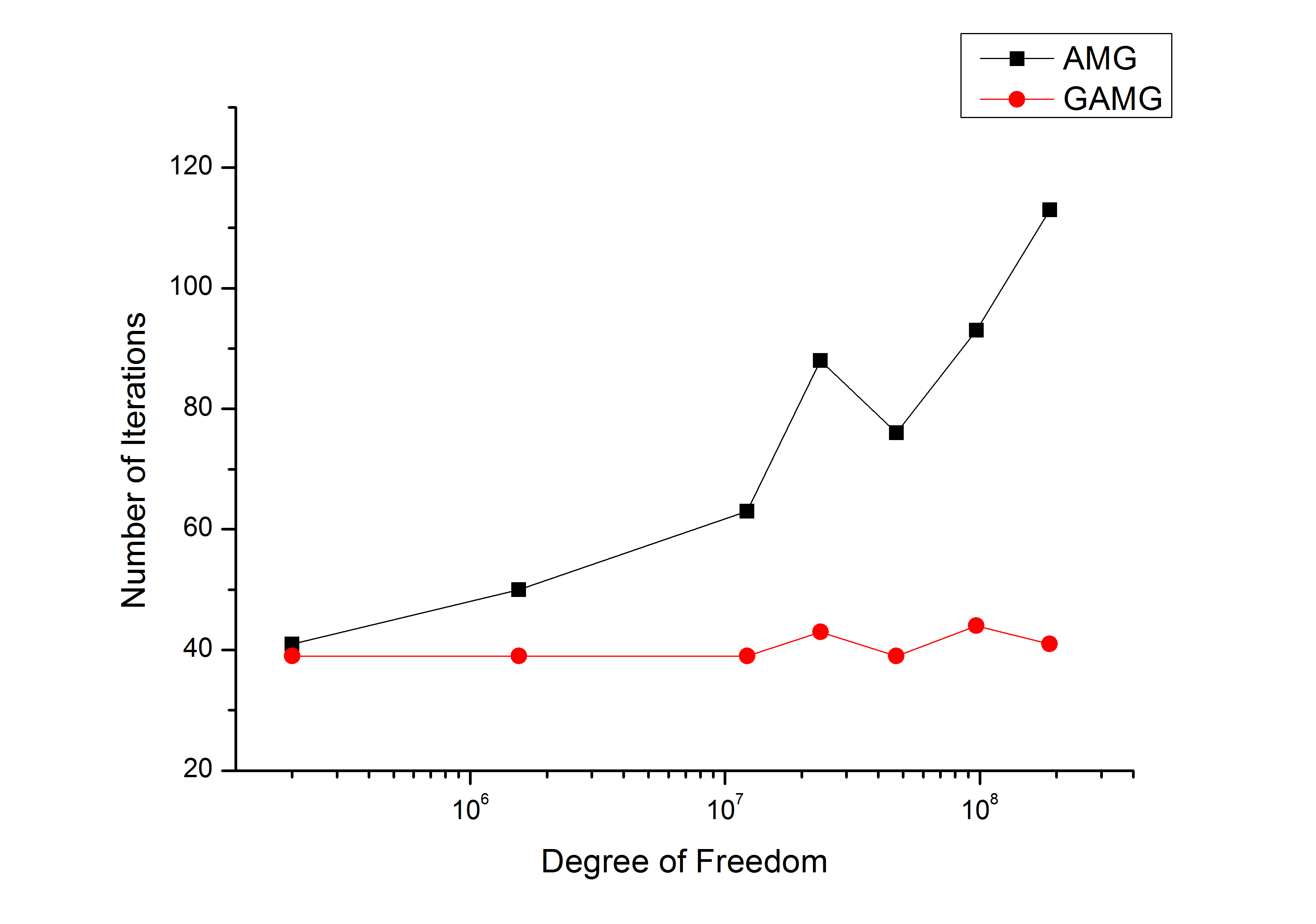}
 }
\subfigure{
   \includegraphics[width=0.475\linewidth] {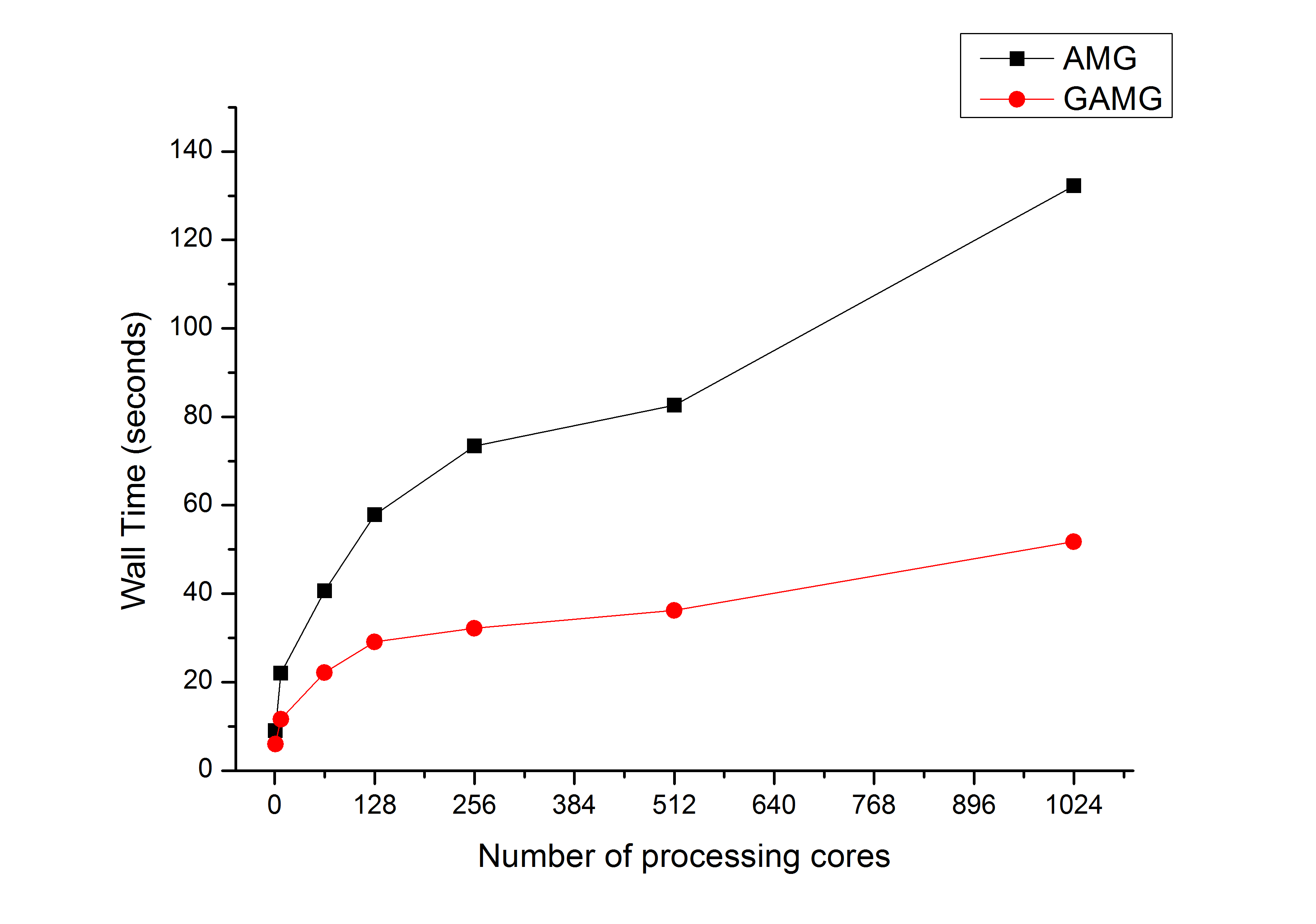}
 }
   \caption{Algorithm optimality and parallel (weak) scalability of AMG and GAMG preconditioned FGMRES methods for the $P^{3,0}-P^{2,0}$ Hood-Taylor mixed finite element.}
   \label{fig:P3-P2}
\end{figure}

\begin{figure}[h!!] 
   \centering
 \subfigure{
   \includegraphics[width=0.475\linewidth] {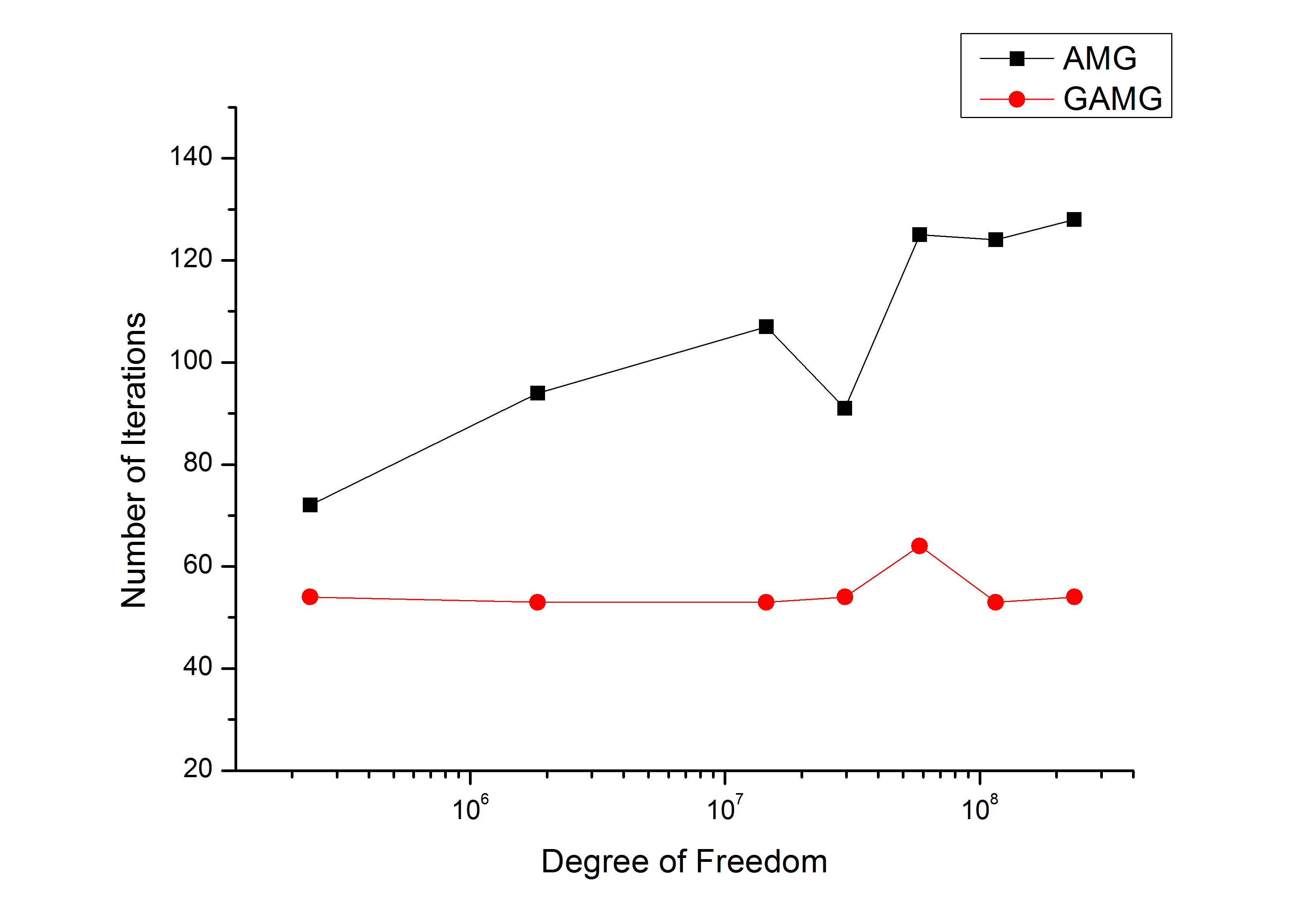}
 }
\subfigure{
   \includegraphics[width=0.475\linewidth] {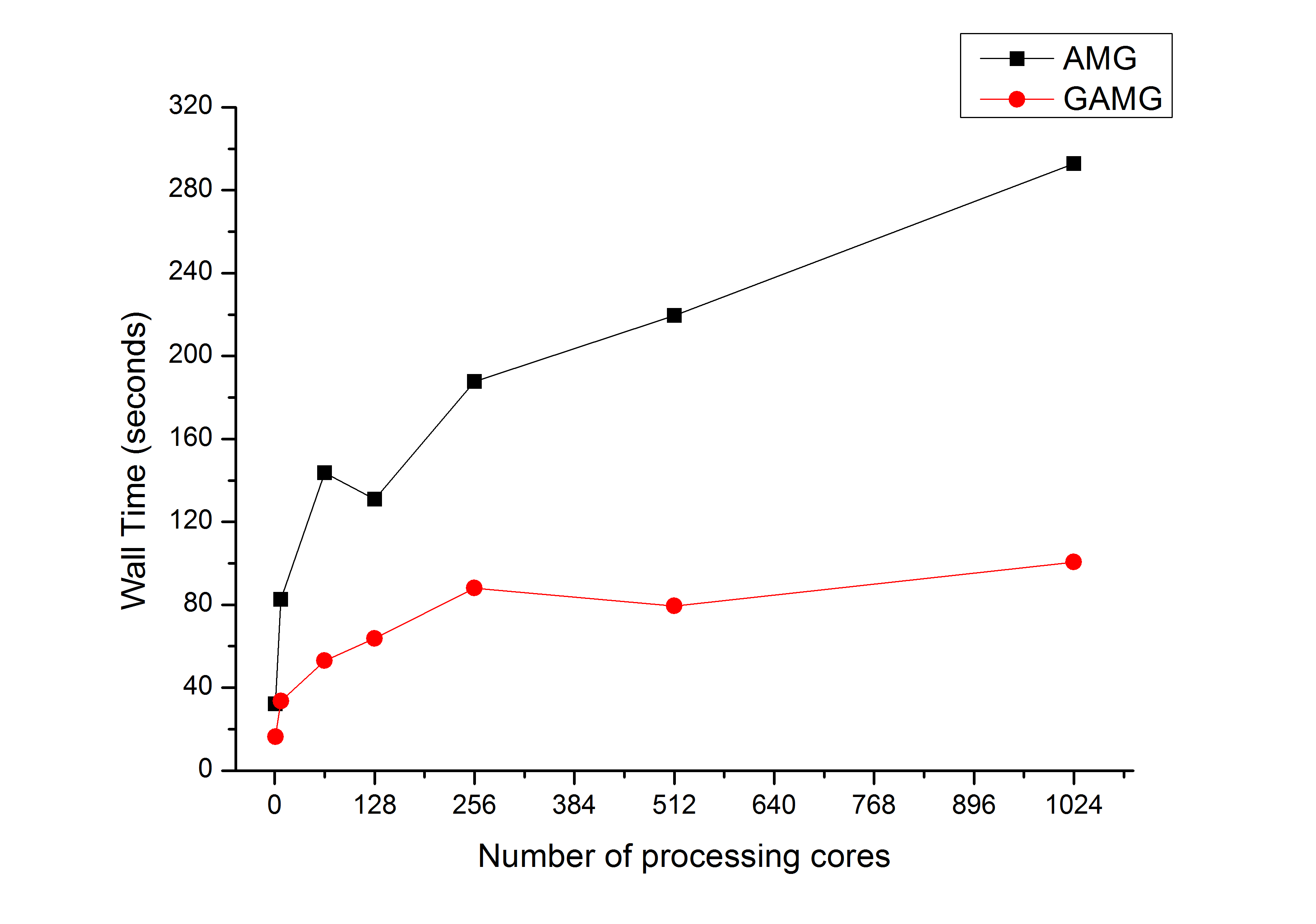}
 }
   \caption{Algorithm optimality and parallel (weak) scalability of AMG and GAMG preconditioned FGMRES methods for the $P^{4,0}-P^{3,0}$ Hood-Taylor mixed finite element.}
   \label{fig:P4-P3}
\end{figure}

\begin{remark}[Intermediate Approximation Spaces]\rm
In this paper, we only use two-level approximation with $P^{1,0}$ finite element space as the coarse level. One can imagine that intermediate (larger) auxiliary spaces could be used to improve performance. Since the convergence rate of the proposed two-level algorithm is optimal (does not deteriorate as size of the problem increases) in our experiments, we decide not to do it in order to keep the implementation as simple as possible.  
\end{remark}


\section{Conclusions}

In this paper, we investigate an auxiliary space preconditioning method for high-order finite element discretizations of the Laplace operator in 3D. 
Modern parallel AMG techniques like PMIS/Extended+i+cc give good parallel scalability, but only if their parameters like strong strength threshold are chosen appropriately. 
On the contrary, the proposed auxiliary space preconditioner is very robust  with respect to coarsening parameters, especially when applied as a building block of Poisson-based solvers for the Stokes equation in 3D.
Furthermore, the proposed method yields smaller operator complexity, which leads to less memory usage and computational complexity.
%

\section*{Acknowledgement}

The authors would like to thank Dr. Tzanio Kolev, Prof. Jinchao Xu, and Prof. Lin-Bo Zhang for their insightful comments and suggestions. Lee has been supported in part by NSF-DMS 091528 and Zhang has been supported in part by NSF-DMS 0915153.


\end{document}